\documentclass[11pt, oneside, a4paper]{article}
\usepackage{geometry}
\title{Drinfeld modules in rank 2 with CM and S-unit j-invariants}
\author{ {Liam {\sc Baker}} \and {Fabien {\sc Pazuki}} \and {Patricio {\sc P{\'e}rez-Pi{\~n}a}} }
\date{} 
\usepackage[english]{babel}
\usepackage[utf8]{inputenc}
\usepackage{lmodern}
\usepackage{amssymb,amscd,amsthm}
\usepackage{mathrsfs}
\usepackage{tikz-cd}
\usepackage{parskip}
\usepackage{mathtools}
\mathtoolsset{centercolon}
\usepackage{enumerate}
\usepackage{graphicx}
\usepackage{csquotes}
\usepackage{comment}
\usepackage{float}
\usepackage[backend=biber, maxnames=99]{biblatex}
\usepackage{titlesec}
\titleformat{\subsection}[runin]{\bfseries}{\thesubsection.}{0.2em}{}[.\hspace{0.4em}-- ]        
\titleformat{\section}{\center\Large\bfseries}{\thesection.}{0.2em}{}[]
\usepackage{url}
\definecolor{BrickRed}{RGB}{153, 0, 0}
\definecolor{TealBlue}{RGB}{84, 181, 183}
\definecolor{Orchid}{RGB}{164,117,181}
\definecolor{Violet}{RGB}{123,59,143}
\usepackage[colorlinks=true, citecolor=Orchid, linkcolor=Violet, urlcolor=TealBlue]{hyperref}
\usepackage{scalerel}
\usepackage{biblatex}
\newtheorem{itheo}{Theorem}

\newtheorem{ipropo}[itheo]{Proposition}

\newtheorem{prop}{Proposition}[section]

\newtheorem{theorem}[prop]{Theorem}

{   \theoremstyle{definition}
    \newtheorem{definition}[prop]{Definition}
    \newtheorem{rem}[prop]{Remark}
    \newtheorem{ex}[prop]{Example}}
\newcommand{\F}{\ensuremath{\mathbb{F}}}

\newcommand{\C}{\ensuremath{\mathbb{C}}}

\renewcommand{\O}{\ensuremath{\mathcal{O}}}

\def\p[#1]_#2{
	\setbox0=\hbox{$\scriptstyle{#2}$}
	\setbox2=\hbox{$\displaystyle{#1}$}
	\setbox4=\hbox{${}'\mathsurround=0pt$}
	\dimen0=.5\wd0 \advance\dimen0 by-.5\wd2
	\ifdim\dimen0>0pt
	\ifdim\dimen0>\wd4 \kern\wd4 \else\kern\dimen0\fi\fi
	\mathop{{#1}'}_{\kern-\wd4 #2}}
\def\prodp_#1{\p[\prod]_{#1}}

\begin{document}
\pagestyle{plain}
    \maketitle

    \centerline{\rule{7cm}{0.5pt}}
	\paragraph{Abstract --}%
    We prove the finiteness of the set of $j$-invariants of Drinfeld modules of rank 2 over $\mathbb{F}_q[T]$ which are CM and $S$-units, for $S$ the infinite set of primes with even degrees.
    The proof is based on the study of ordinary reduction and supersingular reduction of Drinfeld modules, and on the splitting behaviour of primes dividing the difference of two Drinfeld singular moduli.
    We also provide an algorithm to compute a polynomial with coefficients in $\F_q[T]$ and roots the $j$-invariants having CM by a given order, and use it to compute some explicit examples, providing for instance counterexamples to a conjecture of Dorman. For a maximal order $\mathcal{O}$, we prove by a universality argument that our algorithm computes the Hilbert modular polynomial $H_\mathcal{O}$.
	\medskip
	
	\noindent\textit{Keywords:}   
	Drinfeld modules, 
    Complex multiplication, 
    Drinfeld singular moduli, Hilbert polynomials.
    
	\smallskip
	\noindent\textit{2020 Mathematics Subject Classification:}
    {Primary 11G09, 11G15, 11R37, 11Y40, 14G17, 14G40; Secondary 11J93.}
    % %11G09: Drinfel'd modules, higher dimensional motives
    % %11G15: Complex multiplication and moduli of abelian varieties
    % %11R37: Class field theory
    % %14G17: Positive characteristic ground fields in algebraic geometry
    % %14G40: Arithmetic varieties and schemes; Arakelov theory; heights
    % %11J93: Transcendence theory of Drinfel'd and t-modules

    % %11G50: Heights
    % %14K02: isogeny

    \centerline{\rule{7cm}{0.5pt}}
%%%%%%%%%%%%%%%%%%%%%%%%%%%%%%%%%%%%%%%%%%%%%%%%%%%%%%%%%%%%%%%%%%%%%%%%%%%%%%%%%%%%%%%%%%%%%%%%%%%%%%%%%%%%%%%%%%%%%%%%%%%%%%%%%%%%
%%%%%%%%%%%%%%%%%%%%%%%%%%%%%%%%%%%%%%%%%%%%%%%%%%%%%%%%%%%%%%%%%%%%%%%%%%%%%%%%%%%%%%%%%%%%%%%%%%%%%%%%%%%%%%%%%%%%%%%%%%%%%%%%%%%%

\section*{Introduction}
%%%%%%%%%%%%%%%%%%%%%%%%%%%%%%%%%%%%%%%%%%%%%%%%%%%%%%%%%%%%%%%%%%%%%%%%%%%%%%%%%%%%%%%%%%%%%%%%%%%%%%%%%%%%%%%%%%%%%%%%%%%%%%%%%%%%
%%%%%%%%%%%%%%%%%%%%%%%%%%%%%%%%%%%%%%%%%%%%%%%%%%%%%%%%%%%%%%%%%%%%%%%%%%%%%%%%%%%%%%%%%%%%%%%%%%%%%%%%%%%%%%%%%%%%%%%%%%%%%%%%%%%%

The $j$-invariant of an elliptic curve with complex multiplication (CM) is traditionally referred to as a \textit{singular modulus}.
Singular moduli are algebraic integers which play a central role in the class field theory of imaginary quadratic fields, as they generate their Hilbert class fields, see for instance  \cite[Theorem 4.1 page 121]{Sil94}. They are furthermore involved in modern questions in diophantine geometry. Indeed, we know from \cite{And98} that an irreducible plane curve which is not a vertical line, not a horizontal line, and not a modular curve, contains at most finitely many points with coordinates $(j_1,j_2)$, where $j_1$ and $j_2$ are singular moduli.

This result of Andr\'e was made effective in \cite{Kuh12}, and in \cite{BMZ13}, in which the authors also proved that no pair of singular moduli lie on the hyperbola with equation $XY=1$.
Motivated by this example, Masser asked the question of whether a singular modulus can be a unit.
Bilu, Habegger, and K\"uhne \cite{BHK20} proved (extending earlier work of Habegger \cite{Hab15}) that there is no singular modulus which is also a unit.
This was generalised by Campagna in \cite{Cam21} to $S$-unit singular moduli, when $S$ is the infinite set of primes congruent to $1$ modulo~$3$.
Such a spectacular result, however, cannot be expected for simply any set $S$: singular moduli are algebraic integers that are not units, they are thus divisible by primes, so the set of singular moduli which are $S$-units will be non-empty for some sets of places $S$.
By using equidistribution results, Herrero, Menares, and Rivera-Letelier proved nonetheless that for any finite set of places $S$, there are at most finitely many singular moduli that are $S$-units, see \cite{HMRL24}.

Let us now focus on the counterparts of these results in the Drinfeld setting.
The $j$-invariant of a Drinfeld module of rank 2 with CM will be referred to as a \textit{Drinfeld singular modulus}.
These $j$-invariants are also crucial to the description of the Hilbert class field of quadratic CM fields in the function field setting, see for instance \cite[Theorem 7.5.17 page 451]{Papi} and the references therein.

The equivalent of Andr\'e's result is due to Breuer, see \cite{Bre05, Bre07, Bre12} for even more general results.
The finiteness of the set of Drinfeld singular moduli being units was obtained recently in \cite{AABP26}.
More precisely, for each fixed prime power $q\geqslant2$, it is shown in \cite{AABP26} that there are at most finitely many Drinfeld singular moduli that are units.
The actual existence of such elements, and the potential uniformity in $q$, remain intriguing facets of this problem.
The authors of \cite{AABP26} also rule out the existence of pairs of Drinfeld singular moduli on hyperbolas of equations $XY=\gamma$ for any non-zero $\gamma\in{\overline{\F}_q[T]}$ satisfying $\deg \gamma\leqslant q^2-2$.

The first aim of the present article is to extend the study to the case of Drinfeld singular moduli that are also $S$-units, for $S$ a particular set of places of $\F_q[T]$.
We work exclusively with the set $$S_{0}=\{\mathfrak{p}\mid \mathfrak{p}\mbox{ has even degree}\}$$ of places of even degrees, and we do so because we know how to decide whether the Drinfeld module has ordinary or supersingular reduction at these places.
The degree we refer to is the degree in $T$; in particular the set $S_{0}$ is infinite.
Our main result is the following:

\begin{itheo}\label{finiteness}
Let $q\geqslant2$ be a prime power.
Let $S_{0}$ be the set of finite places in $\F_q(T)$ of even degree.
There are at most finitely many Drinfeld singular moduli that are $S_{0}$-units.
\end{itheo}

The fact that $S_{0}$ is infinite is striking, and is in clear analogy with the results of Campagna \cite{Cam21, Cam23}.
Theorem \ref{finiteness} is obtained by reducing the question (via Theorem \ref{dichotomy} and Theorem \ref{bijection} below) to the finiteness of the set of CM units, which allows us to apply the result obtained in \cite{AABP26}.
In analogy with the study of singular moduli of elliptic curves over number fields, the present article plays the role of \cite{Cam21} vis-a-vis \cite{BHK20}.
Whether or not the set of Drinfeld singular moduli that are $S$-units (for an arbitrary finite set of places $S$) is finite is still an open question.
This might require the study of equidistribution properties of Drinfeld modules, and will be the object of future work.

Drinfeld singular moduli that are $S_{0}$-units also satisfy the following property, which is obtained as a consequence of Brown's Theorem 2.8.2 \cite{Bro92}.

\begin{ipropo}\label{conditions}
Let $q$ be a prime power such that $q\equiv 1 \pmod{4}$.
There is no Drinfeld singular modulus $j$ associated with a quadratic order of prime discriminant $D$ which is of odd degree, such that $j$ is also an $S_{0}$-unit.
\end{ipropo}

The discrepancy between the elliptic curve case (no singular modulus is a unit, \cite{BHK20}) and the Drinfeld module case (for each $q\geqslant2$, at most finitely many Drinfeld singular moduli can be units, \cite{AABP26}) comes in part from the fact that there are infinitely many possible base fields $\F_q$, as opposed to the unique field of rational numbers $\mathbb{Q}$.
Proposition \ref{conditions} could be viewed as a first step towards bridging this gap.
Note that in the case where $q$ is even, using Theorem 5.2 in \cite{HsYu98} instead of \cite{Bro92}, one may perform a similar study, which leads to some weaker constraints, as collected in Remark \ref{char2}.
Moreover, when $q$ is odd, we show in Remark \ref{prime D odd degree} that there exists explicit Drinfeld singular moduli associated with prime $D$ of degree 1 which are not $S_0$-units.

Let us now introduce some vocabulary to help us present our next results, which are needed to obtain a proof of Theorem \ref{finiteness}.
Firstly, $A = \F_q[T]$, $k = \F_q(T)$, and a quadratic extension $K/k$ is called imaginary if it is not split at the infinite place $\infty=T^{-1}$.
Secondly, for every Drinfeld singular modulus $j$, we denote by $\mathcal{O}_{j}$ the endomorphism ring of any Drinfeld module $\Phi$ satisfying $j(\Phi)=j$.
We also denote by $K_j$ the fraction field of $\mathcal{O}_j$, which is an imaginary extension of $k$.

\begin{itheo}\label{dichotomy}
    Let $j_1\neq j_2$ be two Drinfeld singular moduli.
    Denote by $\mathfrak{c}_1$ and $\mathfrak{c}_2$ the conductors of $\mathcal{O}_{j_1}$ and $\mathcal{O}_{j_2}$ respectively.
    Suppose that $L$ is a function field containing $j_1$ and $j_2$, and let $\mathfrak{P}$ be a prime in $L$ above a prime $\mathfrak{p}$ in $k$.
    If $j_1\equiv j_2\mod\mathfrak{P}$ then either $K_{j_1}=K_{j_2}$ and $\mathfrak{p}$ divides $\mathfrak{c}_1\mathfrak{c}_2$, or $\mathfrak{p}$ is non-split in both $K_{j_1}$ and $K_{j_2}$.
\end{itheo}

We denote by $\mathcal{J}$ the collection of all Drinfeld singular moduli.
For $\mathfrak{p}$ a finite prime in $k$, we let $\mathcal{J}_\mathfrak{p}\subseteq\mathcal{J}$ be the set of Drinfeld singular moduli $j$ for which $\mathfrak{p}$ splits in $\mathcal{O}_j$.
This condition is equivalent to: $\mathfrak{p}$ splits in $K_j$ and $\mathfrak{p}$ does not divide the conductor of $\O_j$.

Fix a prime $\overline{\mathfrak{P}}$ in $\overline{k}$ above $\mathfrak{p}$ and identify the residue field at $\overline{\mathfrak{P}}$ with $\overline{\F}_\mathfrak{p}$.
Since Drinfeld singular moduli are integral over $A$, reduction mod $\overline{\mathfrak{P}}$ induces a map $\mathrm{red}\colon\mathcal{J}\mapsto \overline{\F}_\mathfrak{p}$.
The definition of ordinary is recalled below in Definition \ref{ord&sup}.
Let us state the following result first published in \cite{BaKo92}.

\begin{itheo}\label{bijection}(Bae and Koo)
    When restricted to $\mathcal{J}_\mathfrak{p}$, the reduction map $\mathrm{red}$ induces a bijection between $\mathcal{J}_\mathfrak{p}$ and the set of ordinary invariants in $\overline{\F}_\mathfrak{p}$. 
\end{itheo}

We include in the sequel a detailed proof of this result.
Theorem \ref{bijection} is used in the proof of Theorem \ref{dichotomy}, which in turn is used in the proof of Theorem \ref{finiteness}. 
The surjectivity of the reduction map follows from a function field version of Deuring's Lifting Theorem, obtained as Theorem 7 in \cite{CojocaruPapikian}.
The injectivity is obtained by considering a simple and transitive action from the class group of associated CM orders, recalled below in Theorem \ref{residualordinarylocus}, and obtained in \cite{Yu}.
Theorem \ref{bijection} and its proof also have links with the work \cite{KaremakerKatenPapikian}, see in particular their section 7, where the authors explain the main ideas for lifting ordinary Drinfeld modules.
Let us also mention that our Proposition \ref{notS0}, used in the proof of Theorem \ref{finiteness}, generalises Proposition 5.6 of \cite{Dorman} (the latter being restricted to odd characteristic and maximal orders).

We also dedicate part of this study to the explicit computations of Drinfeld singular moduli.
This leads naturally to an algorithm used to compute the Hilbert modular polynomial associated with an order $\mathcal{O}$, which we denote by $H_\mathcal{O}$.
Inspired by work of Dummit and Hayes \cite{DuHa94} and by Maciak's \cite[Theorem 4.3.3.]{Mac10}, we prove the following theorem.

\begin{itheo}\label{algorithm}
    Let $\mathcal{O}$ be an $A$-order in a quadratic imaginary extension $K$ of $k$.
    There is an effective algorithm to find a polynomial $P_\mathcal{O}\in \mathcal{O}_K[X]$ satisfying the following properties:
    \begin{enumerate}
        \item The roots of $P_\mathcal{O}$ are the Drinfeld singular moduli with CM by an order containing $\mathcal{O}$.
        \item If the extension $K/k$ is separable (in particular, if $q$ is odd), the coefficients of $P_\mathcal{O}$ are in $A$.
        Moreover,
        \begin{equation}\label{sep product formula}
            P_\mathcal{O}(X)=\prod_{\mathcal{O}\subseteq\mathcal{O}'}H_{\mathcal{O}'}^{r(\mathcal{O}')}(X),
        \end{equation}
        for some positive integers $r(\mathcal{O}')$.
        \item If the extension $K/k$ is inseparable, the characteristic is even and $P_\mathcal{O}^2$ has coefficients in $A$.
        Moreover,
        \begin{equation}\label{insep product formula}
            P_\mathcal{O}^2(X) = \prod_{\mathcal{O}\subseteq\mathcal{O}'} H_{\mathcal{O}'}^{r(\mathcal{O}')}(X),
        \end{equation}
        for some positive integers $r(\mathcal{O}')$.
    \end{enumerate}
In addition, $r(\mathcal{O})=1$ for any maximal order.
In particular, if $\mathcal{O}$ is maximal and we are in the separable case, $P_\mathcal{O}$ is the Hilbert class polynomial $H_\mathcal{O}$.
If $\mathcal{O}$ is maximal and we are in the inseparable case, $P_\mathcal{O}^2$ is the Hilbert class polynomial $H_\mathcal{O}$.
\end{itheo}

The core of the proof is of algorithmic nature and stems from Subsection \ref{subsecalgo}. To conclude that $r(\mathcal{O})=1$ when $\mathcal{O}$ is maximal, we adapt an argument from earlier work of Dummit and Hayes, which relies on the existence of a universal Drinfeld $\mathcal{O}$-module of rank $1$. When $K$ has a degree $1$ prime, a result of Hayes (see Theorem \ref{universality}) guarantees the existence of such an object by working with the notion of \textit{equivalence classes} of Drinfeld modules. For general $K$, we use that such a universal object exists after adding a certain level structure (see Theorem \ref{universalitywithlevel}). We review these results in Section \ref{subsecuniv} after having introduced Drinfeld $\mathcal{O}_K$-modules over $\mathcal{O}_K$-algebras and their group scheme of torsion points. One may suspect that $r(\mathcal{O})=1$ for any order $\mathcal{O}$; see our Remark \ref{rem:r=1} below.

Computing Hilbert modular polynomials is also a classical challenge for elliptic curves, with different successful methods: complex analytic (Enge \cite{Eng09}), $p$-adic (Br\"oker \cite{Bro08}), and CRT (Sutherland \cite{Sut11}).
We underline the fact that Theorem \ref{algorithm} is based on an idea relevant to function field arithmetic, and not inspired by these various existing methods for elliptic curves.

Finally, having the possibility to compute explicitly tables of Drinfeld singular moduli, we studied the following question of Dorman: in \cite{Dorman} page 251, he makes the observation that the norm of a singular modulus seem to always be monic.
In Section 2 of \cite{DuHa94}, Dummit and Hayes prove this fact in the case where the infinite place is ramified (\textit{i.e.} the degree of $D$, the discriminant of the singular modulus, is odd).
In Section \ref{computations} Figure \ref{fig:compu_q5n2}, we study explicit norms of Drinfeld singular moduli and we realize that Dorman's observation is in fact wrong in the inert case (\textit{i.e.} the degree of $D$ is even). Indeed, we have the following examples:
\begin{ipropo}\label{counterex}
    For $q=5$ and $D=2T^2+1$, $D = 2T^2+3$, and $D = 2T^2+4$, the norm of a singular modulus $j$ with discriminant $D$ is equal to $3(T+1)^{12}(T+4)^{12}$, $4T^{12}(T+1)^6(T+4)^6$, and $2(T+2)^{12}(T+3)^{12}$ respectively.
    In particular, the norm of $j$ is not monic in any of these cases. 
\end{ipropo}

The article is organized as follows.
Section \ref{prelim} contains the basic definitions and facts needed in Section \ref{proofs}, where we gather the proofs of Theorem \ref{finiteness}, Proposition \ref{conditions}, Theorem \ref{dichotomy}, and Theorem \ref{bijection}.
Section \ref{computations} shows how to compute Drinfeld singular moduli explicitly by proving Theorem \ref{algorithm}, and gathers a few interesting explicit examples besides the one from Proposition \ref{counterex}.

\section{Preliminaries}\label{prelim}

Let $q\geqslant2$ be a prime power. Let $A=\F_q[T]$, let
$k=\F_q(T)$ be its quotient field, let $k_{\infty}=\F_q((1/T))$ be the completion of $k$ at the infinite place $\infty=1/T$,
and $\mathbb{C}_{\infty}$ be the completion of an algebraic closure of $k_{\infty}$.
If ${\mathcal K}$ is a subfield of $\mathbb{C}_\infty$, we denote by $\overline{{\mathcal K}}$ its algebraic closure in $\mathbb{C}_\infty$, and by ${\mathcal K}^{sep}$ its separable closure in $\overline{{\mathcal K}}$.

We denote by $\vert . \vert$ the absolute value on $\mathbb{C}_{\infty}$ normalized by $\vert T \vert=q$, and
for $z\in\mathbb{C}_{\infty}^{\times}$ we set $\deg z = \log_q\vert z\vert$, where $\log_q$ is the logarithm to the base $q$.
Thus, for any non-zero $a\in A$, we have that $\deg a$ is the degree of the polynomial $a$. 

Let $L$ be a field equipped with an $\F_q$-algebra morphism $\gamma\colon A\to L$.
A Drinfeld $A$-module $\Phi$ of rank $2$ over $L$ is given by a twisted polynomial \[\Phi_T=\gamma(T)+g\tau+\Delta\tau^2\in L\{\tau\},\] where $\Delta\neq0$.
A morphism (or isogeny) $\Phi\to\Phi'$ over $L$ is by definition a polynomial $u\in L\{\tau\}$ such that $u\Phi_T=\Phi'_Tu$.
We denote by $\mathrm{End}_L(\Phi)=\{u\in L\{\tau\}\mid u\Phi_T=\Phi_Tu\}$ the endomorphism ring of $\Phi$.

The $j$-invariant of $\Phi$ is then given by the quotient \[j(\Phi)=\frac{g^{q+1}}{\Delta}.\] It characterises the isomorphism class of the Drinfeld module $\Phi$ over $\overline{L}$.

As in the case of elliptic curves, Drinfeld modules of rank 2 over $\mathbb{C}_\infty$ can be described in terms of lattices as follows.
A rank 2 lattice $\Lambda\subseteq\mathbb{C}_\infty$ is an $A$-module of rank $2$.
Two lattices $\Lambda$ and $\Lambda'$ are \textit{equivalent} if $\Lambda'=c\Lambda$ for some $c\in\mathbb{C}_\infty^\times$.
The exponential function \[e_\Lambda(z)=z\prod_{\lambda\in\Lambda\smallsetminus\{0\}}\left(1-\frac{z}{\lambda}\right)\] induces an analytic isomorphism $\mathbb{C}_\infty/\Lambda\cong \mathbb{C}_\infty$ and there exists a unique polynomial \[\Phi_T^{\Lambda}=T+g(\Lambda)\tau+\Delta(\Lambda)\tau^2\] with coefficients in $\mathbb{C}_\infty$ such that multiplication by $T$ in $\mathbb{C}_\infty/\Lambda$ corresponds to the action of $\Phi_T^{\Lambda}$ on $\mathbb{C}_\infty$ under the previous isomorphism.
The assignment $\Lambda\mapsto \Phi_T^\Lambda$ induces a bijection between the set of equivalence classes of lattices and isomorphism classes of Drinfeld module of rank 2 over $\C_\infty$.

Every lattice of rank 2 in $\mathbb{C}_\infty$ is equivalent to a lattice of the form $\Lambda_z:=A+Az$ for some $z\in\mathbb{C}_\infty\smallsetminus k_\infty$ and $\Lambda_z$ is equivalent to $\Lambda_{z'}$ if and only if $z'=\frac{az+b}{cz+d}$ for some $\begin{psmallmatrix} a&b \\ c&d \end{psmallmatrix} \in \mathrm{GL}_2(A)$.
We put $\Phi_T^{(z)}:=\Phi_T^{\Lambda_z}$ and $j(z):=j\big(\Phi_T^{(z)}\big)$.

If $\Phi$ is a Drinfeld module over $\mathbb{C}_\infty$ isomorphic to $\Phi_T^\Lambda$, we identify $\mathrm{End}_{\mathbb{C}_{\infty}}(\Phi)$ with the subring
$\{\lambda\in\mathbb{C}_{\infty}\mid \lambda\Lambda\subset\Lambda\}$ of $\mathbb{C}_{\infty}$.
A Drinfeld module $\Phi$ of rank $2$ is said to have complex multiplication (or, in short, to be CM) if $\mathrm{End}_{\mathbb{C}_{\infty}}(\Phi)\not=A$.
In this case, the ring $\mathrm{End}_{\mathbb{C}_{\infty}}(\Phi)$ is an $A$-order in an imaginary quadratic extension of $k$; in particular, it is a free $A$-module of rank $2$.
We will say that $\alpha\in\mathbb{C}_{\infty}$ is a \emph{singular modulus} if there exists a CM Drinfeld module $\Phi$ of rank $2$ such that $j(\Phi)=\alpha$.
It is known (see \cite[Theorem 7.5.17, page 451]{Papi}) that Drinfeld singular moduli are integral over $A$.

Let $\mathcal{O}$ be an order in a quadratic imaginary extension $K/k$.
Recall that the class group $\mathrm{Cl}(\mathcal{O})$ is defined as the quotient between the group of proper fractional $\mathcal{O}$-ideals and the subgroup of principal $\mathcal{O}$-ideals.
Viewing $K\subseteq\mathbb{C}_\infty$, we can consider each $\mathcal{O}$-ideal $\mathfrak{a}$ as a rank $2$ lattice in $\mathbb{C}_\infty$. 

\begin{prop}\label{CMbijection}
    The assignment $\mathfrak{a}\mapsto\Phi^{\mathfrak{a}}$ induces a bijection between the ideal class group $\mathrm{Cl}(\mathcal{O})$ of $\mathcal{O}$ and the set of isomorphism classes of Drinfeld $A$-modules of rank $2$ over $\mathbb{C}_\infty$ having CM by $\mathcal{O}$.
\end{prop}
\begin{proof}
    Take two $\mathcal{O}$-ideals $\mathfrak{a}$ and $\mathfrak{b}$.
    If $\Phi^\mathfrak{a}$ and $\Phi^\mathfrak{b}$ are isomorphic, then $\mathfrak{b}=c\mathfrak{a}$ and necessarily $c\in K^\times$ so they define the same element in $\mathrm{Cl}(\mathcal{O})$.
    This shows that the assignment is injective.
    A Drinfeld module $\Phi\cong \Phi^{\Lambda_z}$ has CM by $\mathcal{O}$ if and only if $k(z)/k$ is a quadratic extension and
    $$\mathcal{O} = \{\lambda\in\mathbb{C}_{\infty}\mid \lambda\Lambda\subseteq\Lambda\} = \{\lambda\in k(z)\mid \lambda \Lambda_z\subseteq\Lambda_z\}.$$
    In particular, $\Lambda_z$ is a proper $\mathcal{O}$-ideal which implies surjectivity. 
\end{proof}

\begin{rem}
    Actually, \cite[Theorem 7.5.8]{Papi} and \cite[Remark 7.5.7]{Papi} show that there exists a simple and transitive action of $\mathrm{Cl}(\mathcal{O})$ over the set of isomorphism classes of Drinfeld A-modules of rank 2 over $\mathbb{C}_\infty$ having CM by $\mathcal{O}$.
    Having a bijection is enough for our purposes. 
\end{rem}

For $\mathfrak{P}$ be a finite prime of $L/k$, we denote the completion of $L$ at $\mathfrak{P}$ by $L_\mathfrak{P}$.
Let $\mathcal{O}_{L_\mathfrak{P}}$ be the ring of integers of $L_\mathfrak{P}$ and $\F_\mathfrak{P}$ the residue field at $\mathfrak{P}$.
If $\Phi$ is a Drinfeld module over $L/k$ such that $\Phi_T\in \mathcal{O}_{L_\mathfrak{P}}\{\tau\}$, we let $\overline{\Phi_T}\in \F_{\mathfrak{P}}\{\tau\}$ be the image of $\Phi_T$ under the reduction $\mathcal{O}_{L_\mathfrak{P}}\{\tau\}\to\F_{\mathfrak{P}}\{\tau\}$.

\begin{definition}(Definition 6.1.3 from \cite{Papi})
     A Drinfeld module $\Phi$ of rank $2$ defined over $L$ has good reduction at a finite prime $\mathfrak{P}$ of $L$ if there is a $c\in{L}_{\mathfrak{P}}^*$ such that $\Psi_T:=c^{-1} \Phi_T c$ belongs to $\mathcal{O}_{L_\mathfrak{P}}\{\tau\}$ and $\overline{\Psi}_T$ defines a Drinfeld module of rank $2$ over $\F_\mathfrak{P}$.
     We say that $\Phi$ has potentially good reduction at $\mathfrak{P}$ if there exists a finite extension $F/L$ and a prime $\mathfrak{P}'$ of $F$ above $\mathfrak{P}$ such that $\Phi$ has good reduction at $\mathfrak{P}'$ when considered as a Drinfeld module over $F$.
\end{definition}

If a CM Drinfeld module $\Phi$ of rank $2$ is defined over $L/k$, then $\Phi$ has potentially good reduction at every finite prime of $L$.
This follows from \cite[Proposition 4.2]{Gekeler} and the integrality of $j(\Phi)$.

\begin{definition}\label{ord&sup}
    A Drinfeld $A$-module $\overline{\Phi}$ of rank $2$ over a finite field $\F$ is said to be ordinary if $\mathrm{End}_{\overline{\F}}(\overline{\Phi})$ is an $A$-order in a quadratic imaginary field.
    If this is not the case, $\overline{\Phi}$ is said to be supersingular and $\mathrm{End}_{\overline{\F}}(\overline{\Phi})$ is an order inside a quaternion algebra over $k$.
\end{definition}

\begin{rem}
In view of \cite[Corollary 4.1.1]{Papi}, Definition \ref{ord&sup} is equivalent to  \cite[Definition 4.1.9]{Papi}. 
\end{rem}

\section{Gallery of proofs}\label{proofs}

Let $K/k$ be a quadratic extension and $\mathcal{O}\subseteq\mathcal{O}_K$ an order.
The \textit{conductor} of $\mathcal{O}$ is the unique ideal $\mathfrak{c}$ such that $\mathcal{O}=A+\mathfrak{c}\mathcal{O}_K$.

Let us start with a useful proposition.

\begin{prop}\label{redtype}
Let $\Phi$ be a Drinfeld module with CM by an order of conductor $\mathfrak{c}$ in $K$.
Assume that $\Phi$ is defined over $L/k$ and $\mathfrak{P}$ is a prime in $L$ of good reduction for $\Phi$.
Then the reduction map induces an injective homomorphism $\mathrm{End}_L(\Phi) \to \mathrm{End}_{\F_\mathfrak{P}}(\overline{\Phi})$.
Let $\mathfrak{p}=\mathfrak{P}\cap A$, then
\begin{enumerate}
    \item\label{redtype1} $\overline{\Phi}$ is supersingular if and only if $\mathfrak{p}$ does not split in $K$.
    \item\label{redtype2} Assume that $\mathfrak{p}$ splits in $K$ and write $\mathfrak{c}=\mathfrak{c}'\mathfrak{p}^{n}$ for some $n\geqslant0$ with $\mathfrak{c}$ coprime to $\mathfrak{p}$.
    Then $\overline{\Phi}$ is ordinary and it has CM by the order of conductor $\mathfrak{c}'$ in $K$.
\end{enumerate}
\end{prop}

\begin{proof}
    This is a combination of Corollary 6.1.12, Corollary 4.1.11, and Exercise 7.5.7 in \cite{Papi}.
\end{proof}

Let $\overline{\Phi}$ be a Drinfeld $A$-module of rank $2$ over a finite field $\F$.
Given a proper ideal $I$ of $\mathrm{End}(\overline{\Phi})$, there exists an element $u_I\in\F\{\tau\}$ such that the left ideal $\F\{\tau\}I$ equals $\F\{\tau\}u_I$.
This is a consequence of the right division algorithm for $\F\{\tau\}$.
A direct computation shows that $\F\{\tau\}u_I\overline{\Phi}_T\subseteq \F\{\tau\}u_I$.
We denote by $I*\overline{\Phi}$ the Drinfeld module given by $u_I\overline{\Phi}_Tu_I^{-1}\in \F\{\tau\}$. 

\begin{theorem}\label{residualordinarylocus}
    Assume that $\overline{\Phi}$ is ordinary and let $\mathcal{O}=\mathrm{End}_{\F}(\overline{\Phi})$.
    The assignment $I\mapsto I*\overline{\Phi}$ defines a simple and transitive action of the class group $\mathrm{Cl}(\mathcal{O})$ on the isomorphism classes of Drinfeld $A$-module of rank 2 over $\F$ with endomorphism ring isomorphic to $\mathcal{O}$.
\end{theorem}
\begin{proof}
    This corresponds to case (1) in \cite[Corollary p.170]{Yu}, case (1) referring to the notation in \cite[Proposition 4 p.169]{Yu}.
\end{proof}

\begin{rem}
    In \cite{Yu}, the construction of $I*\overline{\Phi}$ is considered for any ideal of $\mathrm{End}(\overline{\Phi})$, not necessarily proper ideals.
    In this generality, the endomorphism ring of $I*\overline{\Phi}$ agrees with the ring of multipliers of $I$.
    Beyond the rank 2 case, this construction is also considered in \cite{KaremakerKatenPapikian} for the case when $\mathrm{End}(\overline{\Phi})$ is commutative, where one has to distinguish ideals from \textit{kernel ideals} in order to keep the property about the endomorphism ring of $I*\overline{\Phi}$ (see Lemma 4.2 in \cite{KaremakerKatenPapikian}).
    When $\mathrm{End}(\overline{\Phi})$ is a Gorenstein ring, there is no distinction between ideals and kernel ideals (see Proposition 4.5 in \cite{KaremakerKatenPapikian}).
    Observe that in Theorem \ref{residualordinarylocus} we have that $\mathrm{End}(\overline{\Phi})$ is a quadratic $A$-order and these are always Gorenstein by \cite[Section 5.1]{JenTho15}. 
\end{rem}

\subsection{Proof of Theorem \ref{dichotomy} and Theorem \ref{bijection}}

We start with proving Theorem \ref{bijection}, and deduce afterwards Theorem \ref{dichotomy}.

\begin{proof}[Proof of Theorem \ref{bijection}]
    By Proposition \ref{redtype} item \ref{redtype2}, if $j$ reduces to $\overline{j}$, then $\mathcal{O}_j=A+\mathfrak{p}^n\mathcal{O}_{\overline{j}}$ for some $n\geqslant0$.
    Therefore, $j$ belongs to $\mathcal{J}_\mathfrak{p}$ if and only if $\mathcal{O}_{j}=\mathcal{O}_{\overline{j}}$.
    The surjectivity of the reduction map follows from a function field version of Deuring's Lifting Theorem (see \cite[Theorem 7]{CojocaruPapikian}).
    Take $\overline{j}$ ordinary in $\overline{\F}_\mathfrak{p}$ and $\alpha_0$ in $\mathcal{O}_{\overline{j}}$ such that $\mathcal{O}_{\overline{j}}=A[\alpha_0]$.
    There exists a lift $(j,\alpha)$ of the pair $(\overline{j},\alpha_0)$ to $A$-characteristic zero.
    Since $\alpha\in\mathcal{O}_j\subseteq \mathcal{O}_{\overline{j}}$ and $\alpha$ reduces to $\alpha_0$, we can conclude that $\mathcal{O}_{j}=\mathcal{O}_{\overline{j}}$ and so $j$ belongs to $\mathcal{J}_{\mathfrak{p}}$.
    Finally, restricting to the subset of $\mathcal{J}_{\mathfrak{p}}$ with a given endomorphism ring $\mathcal{O}$, we obtain a surjective map onto the ordinary Drinfeld singular moduli with endomorphism ring $\mathcal{O}$.
    By Proposition \ref{CMbijection} and Theorem \ref{residualordinarylocus}, both sets have cardinality $\#\mathrm{Cl}(\mathcal{O})$ and so this map is injective. 
\end{proof}

This is enough to obtain the following proof of Theorem \ref{dichotomy}.

\begin{proof}[Proof of Theorem \ref{dichotomy}]
    The congruence $j_1\equiv j_2\mod\mathfrak{P}$ implies that $\mathcal{O}_{\overline{j_1}} \cong \mathcal{O}_{\overline{j_2}}$.
    This order is of rank $4$ if and only if $\mathfrak{p}$ does not split in $K_{j_1}$ and $K_{j_2}$ by Proposition \ref{redtype} item \ref{redtype1}.
    Therefore it is enough to consider the case when $\mathfrak{p}$ splits in $K_{j_1}$ and $K_{j_2}$ and the endomorphism ring has rank $2$.
    Since $K_{j_i}$ is the fraction field of $\mathcal{O}_{\overline{j_i}}$, we must have that $K_{j_1} = K_{j_2}$, and from now on we will simply denote this field by $K$.
    Write $\mathcal{O}_{j_i} = A+\mathfrak{c}_i\mathcal{O}_{K}$ and $\mathfrak{c}_i = \mathfrak{p}^{n_i}\mathfrak{c}_i'$ with $\mathfrak{c}_i'$ coprime to $\mathfrak{p}$.
    Then by Proposition \ref{redtype} item \ref{redtype2}, $\mathcal{O}_{\overline{j_i}} = A+\mathfrak{c}_i'\mathcal{O}_{K}$ from where $\mathfrak{c}_1'=\mathfrak{c}_2'$.
    Finally, suppose that $\mathfrak{p}$ does not divide $\mathfrak{c}_1\mathfrak{c}_2$.
    Then $\mathfrak{c}_1 = \mathfrak{c}_2 = \mathfrak{c}_1' = \mathfrak{c}_2'$ and Theorem \ref{bijection} implies that $j_1 = j_2$, which is a contradiction.
\end{proof}

\subsection{Proof of Theorem \ref{finiteness}}

The singular modulus $j_0=0$ is associated with the imaginary quadratic field $\F_{q^2}(T)$, as can be seen by direct computation with the equation $\Phi_T=T+\tau^2$ for instance.
The extension $\F_{q^2}(T)/\F_{q}(T)$ is unramified at all places and a prime $\mathfrak{p}$ splits if and only if the degree of $\mathfrak{p}$ is even.

Recall that $S_{0}$ is the infinite set of finite places in $\F_q(T)$ given by \[S_{0}=\{\mathfrak{p}\mid \mathfrak{p}\mbox{ has even degree}\}.\]
The following two propositions show that any singular modulus which is an $S_{0}$-unit is in fact a unit.

\begin{prop}\label{notS0}
    Let $j$ be a Drinfeld singular modulus with imaginary quadratic field $K\neq\F_{q^2}(T)$.
    Let $S$ be any set of places of $k=\mathbb{F}_q(T)$.
If $j$ is an $S$-unit, then it is an $S\backslash S_0$-unit.  
\end{prop}
\begin{proof}
    Let $\mathfrak{P}$ be a prime in $k(j)$ above $\mathfrak{p}$ and assume that $j\equiv j_0\mod \mathfrak{P}$.
    By our assumption and Theorem \ref{dichotomy}, $\mathfrak{p}$ must be inert in both $K$ and $\F_{q^2}(T)$.
    This implies that $\mathfrak{p}$ is not in $S_{0}$.
    Thus $j$ is an $S\backslash S_0$-unit.
\end{proof}

\begin{prop}\label{not0bis}
    Let $j$ be a Drinfeld singular modulus with associated imaginary quadratic field $\F_{q^2}(T)$.
    Then for every prime $\mathfrak{p}$ of degree $1$, there exists some prime $\mathfrak{P}$ above $\mathfrak{p}$ in $k(j)$ dividing $j$.
    In particular, $j$ is not an $S_{0}$-unit.
\end{prop}
\begin{proof}
    Since $\mathfrak{p}$ is of odd degree, it is inert in $\F_{q^2}(T)$.
    In particular, $j$ has supersingular reduction at $\mathfrak{p}$ (Proposition \ref{redtype} item \ref{redtype1}).
    By \cite[Corollary 4.4.12]{Papi}, the only supersingular Drinfeld module of rank $2$ in characteristic $\mathfrak{p}$ is $0$.
    Therefore, $j$ is divisible by some prime $\mathfrak{P}$ above $\mathfrak{p}$ in $k(j)$.
\end{proof}

From Propositions \ref{notS0} and \ref{not0bis} we conclude that for a Drinfeld singular modulus, being an $S_{0}$-unit is equivalent to being a unit.
Therefore, Theorem \ref{finiteness} is a consequence of \cite[Theorem 1.2]{AABP26}.

\begin{rem}
    Proposition \ref{notS0} generalises Proposition 5.6 of \cite{Dorman} (which is restricted to odd charactersitic and maximal orders). Indeed, if $j$ has CM by the maximal order in $\mathbb{F}_{q^2}[T]$, then $j=0$.
\end{rem}

Let us now add that one can construct infinitely many examples of singular moduli that are divisible by primes of both even and odd degree.
Indeed, consider $\mathfrak{q}$ a prime of even degree in the following complement to Proposition \ref{not0bis}.

\begin{prop}
    Let $n\geqslant1$ be an integer.
    Let $j$ be a Drinfeld singular modulus, with CM by the order of conductor $\mathfrak{q}^n$ in $\mathbb{F}_{q^2}(T)$.
    Then for each prime $\mathfrak{p}$ of degree $1$ or equal to $\mathfrak{q}$, there exists some prime $\mathfrak{P}$ above $\mathfrak{p}$ in $k(j)$ dividing $j$.
\end{prop}
\begin{proof}
    For $\mathfrak{p}$ of degree $1$ we know that this holds from Proposition \ref{not0bis}.
    Working with $\mathfrak{p}=\mathfrak{q}$, from Proposition \ref{redtype} we know that the reduction of $j$ modulo $\mathfrak{p}$ has CM by the maximal order in $\mathbb{F}_{q^2}(T)$.
    Thus the reduction must be zero and therefore $j$ is divisible by some prime in $k(j)$ above $\mathfrak{q}$.
\end{proof}

\subsection{Proof of Proposition \ref{conditions}}

We will now use Brown's Theorem 2.8.2 in \cite{Bro92}, as $q$ is assumed to be odd, to obtain a proof of Proposition \ref{conditions}.

\begin{proof}[Proof of Proposition \ref{conditions}]
In this case we have that $K/k$ is separable, and in that case $[k(j):k]=[K(j):K]$ by \cite[Lemma 2.1]{AABP26}: let us call this degree $h$.
The crucial question is: when is $h$ odd?
The degree $h$ is odd if and only if $D$ is of odd degree and a prime discriminant.
This follows from \cite[Proposition 14.7]{Rosen} and \cite[Section 11]{Artin}.

Let us now recall Theorem 2.8.2 from \cite{Bro92}:
\begin{theorem} \label{Bro}(Brown)
Let $q$ be odd and $z\in \mathbb{C}_\infty\smallsetminus k_\infty$ such that $j(z)$ has discriminant $D$.
Let $n\geqslant 0$ be the smallest integer such that $n\geqslant \log _q |z|$.
\begin{enumerate}
    \item If $\deg D$ is odd, then $n\geqslant 1$ and $\vert j(z)\vert = q^{\frac{q+1}{2} q^n}$.
    \item If $\deg D$ is even, then:
    \begin{enumerate}
        \item if $n\geqslant 1$, we have $|j(z)|= q^{q^{n+1}}$; 
        \item if $n=0$, then there exists a unique $e\in \mathbb F_{q^2}\setminus \mathbb F_q$ such that $|z-e| < 1$, and we have
        $|j(z)| = q^q |z-e|^{q+1}$.
    \end{enumerate}
\end{enumerate}
\end{theorem}

% z is the root of a quadratic equation with discriminant D
% look at Fabien's paper for the set of values of z, simpler when q is odd
% n <= deg D / 2

Assume now that $j$ is an $S_{0}$-unit.
From \cite[Chapter VI paragraph 8 no.7 Proposition 10]{Bourbaki}, it follows that $|N_{k(j)/j}(j)|=|j|^{[k(j):k]}=|j|^h$.
Since $j$ is an $S_{0}$-unit if and only if the only primes dividing $N_{k(j)/j}(j)$ are in $S_{0}$, we have that $\log_q|j|^h$ is even.
By Brown's Theorem, when $\deg{D}$ is odd this is equivalent to $h\frac{q+1}{2}q^n$ being even.
We can conclude that if $q$ and $\deg{D}$ are odd, then $\frac{q+1}{2}h$ must be even but this is not possible with the assumptions of Proposition \ref{conditions}.
This completes the proof.
\end{proof}

%If the degree of $D$ is even, in both subcases $h$ must be even. 

\begin{rem}\label{char2}
    For $q$ even, the conclusions are not as useful.
    Applying Theorem 5.2 of \cite{HsYu98} instead of Brown's Theorem, we may follow the same line of thought.
    It is still true that $h$ is even when the extension is inert, thus one can only say something non-trivial when $\deg{D}$ is odd or the extension is inseparable.
    Assuming $\deg{D}$ odd, one can say something non-trivial only for $q^n/2$ odd, which implies $q=2$ and $n=1$.
    The inseparable case only treats one discriminant.
    Moreover, by \cite[remark 2.5]{AABP26}, in this case $[k(j):k]=2h$ so we have $\log_q|j|^{2h}$ even, which is always true.
\end{rem}

\section{Computing Drinfeld singular moduli}\label{computations}

The goal of this section is to show that one can compute explicitly the $j$-invariant of a rank 2 Drinfeld module with prescribed complex multiplication.
This helps us understand what possible set $S$ we may consider to find explicit Drinfeld singular moduli which are $S$-units.

\subsection{Imaginary quadratic orders}
Before describing the algorithm, let us start by describing explicit generators for quadratic $A$-orders arising as endomorphism rings of rank-2 CM Drinfeld modules.
In odd characteristic, we follow \cite[Section 2.2.2]{AABP26}: every imaginary quadratic extension $K$ of $k$ is separable
and has the form $K=k(\sqrt{D})$,
where $D\in A$ is not a square in $k$.
For $D\in A$, $D\neq0$, the field $K=k(\sqrt{D})$ is an imaginary quadratic extension of $k$ if and only if the following
condition holds:
\begin{equation}
\begin{cases}\label{cond-imaginary}
\text{(i) $\deg D$ is odd}\\
\text{or}  \\
\text{(ii) $\deg D$ is even and the leading coefficient of $D$ is not a square in $\mathbb{F}_q^{\times}$}.
\end{cases}
\end{equation}
By Proposition 14.6 in \cite{Rosen}, we know that $$(\deg D\, \textrm{is odd}) \iff (\infty=1/T \,\textrm{is ramified in}\, K/k)$$ and that $$(\deg D\, \textrm{is even and the leading coefficient of $D$ is not a square in}\, \mathbb{F}_q^{\times}) \iff (\infty=1/T\,\textrm{is inert in}\, K/k).$$  

When Condition (\ref{cond-imaginary}) is satisfied, the $A$-module
\begin{equation}\label{Hilbert odd}
    \mathcal{O}=A[\sqrt{D}]=A+A\sqrt{D}
\end{equation}
is the unique order in $K$ whose discriminant is $D$. 

In characteristic 2, every order $\mathcal{O}$ in an imaginary quadratic separable extension of $k$ is obtained as follows, see \cite[Section 2.2.2]{AABP26} or \cite[Section 2]{Chen08}.
Let $D=\prod_{i=1}^sP_i^{2n_i}$ be a (possibly empty) product with each $P_i$ a monic irreducible element in $A$.
Consider $\mathrm{rad}(D):=\prod P_i$ and $\sqrt{D}:=\prod P_i^{n_i}$ and let  $C=D/\mathrm{rad}(D)$.
Take $B\in A$ such that $\deg B\geqslant\deg C$ and choose $\xi$ a root of \[X^2+X+B/C.\]
Then $K=k(\xi)$ is a separable imaginary quadratic extension of $k$ and its maximal order is $A[\sqrt{D}\xi]$ (of discriminant $D$).
For every monic $f$, 
\begin{equation}\label{Hilbert even}
    \mathcal{O}=A[f\sqrt{D}\xi] = A + Af\sqrt{D}\xi
\end{equation}
is the unique suborder of conductor $f$.
The minimal polynomial of $f\sqrt{D}\xi$ is
\[X^2+f\sqrt{D} X+f^2\mathrm{rad}(D)B.\]

In characteristic $2$, we also have the unique imaginary quadratic inseparable extension $K=k(\sqrt{T})$.
Its maximal order is $A[\sqrt{T}]=\mathbb{F}_q[\sqrt{T}]$ and every suborder is of the form \begin{equation}\label{Hilbert insep}\mathcal{O}=A[f\sqrt{T}],\end{equation} for some $f\in A$.

\subsection{Universal Drinfeld module}\label{subsecuniv}

In this subsection, we review the representability of certain functors associated to Drinfeld modules of rank 1, which will help in the proof of Theorem \ref{algorithm}.
We start by extending our defintion of a Drinfeld module.

In this subsection, $K$ denotes an imaginary quadratic extension of $k$ or $k$ itself.
We use $\mathcal{O}_K$ for its maximal $A$-order.
If $K_\infty=K\otimes k_\infty$, we denote by $d_\infty$ the degree the inertia degree of the field extension $K_\infty/k_\infty$.
The normalized valuation of $K_\infty$ will be denoted by $v_{\infty}$. 

Let $R$ be an $\mathcal{O}_K$-algebra with structure map $\delta_R\colon\mathcal{O}_K\to R$.
Consider the natural inclusion $i\colon R\to R\{\tau\}$ and the map $\mathcal{D}\colon R\{\tau\}\to R$ defined by taking the constant term of twisted polynomials.
A Drinfeld $\mathcal{O}_K$-module over $R$ is a ring morphism $\Phi\colon \mathcal{O}_K\to R\{\tau\}$ such that $\mathcal{D}\circ\phi=\delta_K$, $\Phi\neq i\circ\delta_K$, and the leading term of each non-zero polynomial in the image of $\Phi$ is a unit in $R$.
We say that $\Phi$ is of rank $r$ if
\[\deg\Phi_x=-rd_\infty v_\infty(x)=r[K_\infty\colon k_\infty]\log_q|x|\]
for all $x$ in $\mathcal{O}$.
We write \[\Phi_x=\delta_K(x)+\sum_{n=1}^{\deg\Phi_x}c_n(\Phi,x)\tau^n.\]

If $K=k$, $R$ is a field, and $r=2$, this definition agrees with the one given in Section \ref{prelim} since any morphism $A\to R$ is determined by $\Phi_T\colon=\Phi(T)$.
We can extend the definition of the $j$-invariant by $j(\Phi)=\frac{c_1(\Phi,T)^{q+1}}{c_2(\Phi,T)}\in R$.

Let $\Phi$ and $\Phi'$ be two Drinfeld $\mathcal{O}_K$-modules of rank 1.
We say that $\Phi$ and $\Phi'$ are \textit{equivalent} if there exists $w\in R^\times$ such that
\[ c_n(\Phi',x) = c_n(\Phi,x)w^{\frac{q^n-1}{q-1}} \quad\text{for each}\ n.\]

Suppose that $K/k$ is a quadratic extension; then $\Phi\mid_A$ and $\Phi'\mid_A$ are two Drinfeld $A$-modules of rank 2, and if $\Phi$ and $\Phi'$ are equivalent, then $c_1(\Phi',T)=c_1(\Phi,T)w$ and $c_2(\Phi',x)=c_2(\Phi,x)w^{q+1}$.
In particular, $\Phi\mid_A$ and $\Phi'\mid_A$ have the same $j$-invariant.

Let $F_1$ be the functor which associates to every $\mathcal{O}_K$-algebra $R$ the set of equivalence classes of Drinfeld $\mathcal{O}_K$-modules of rank $1$ over $R$.
Let $H$ be the Hilbert class field of $K$ and $\mathcal{O}_H$ its maximal $A$-order.

\begin{theorem}[Hayes]\label{universality}
Suppose that $K$ has a prime divisor of degree $1$, then $\mathcal{O}_H$ represents the functor $F_1$. 
\end{theorem}
\begin{proof}
    See \cite[Appendix]{Hayes79}.
\end{proof}

If $K$ does not have a prime divisor of degree 1, we can still construct a representable functor by adding level structures.
Let $I$ be an $\mathcal{O}_K$-ideal and $\Phi$ a Drinfeld $\mathcal{O}_K$-module of rank 1 over $R$.
We denote by $\Phi[I]$ the group scheme of $I$-division points of $\Phi$.
For any $R$-algebra $S$,
\[ \Phi[I](S) = \{s\in S\mid \Phi_\alpha(s)=0\ \text{for every}\ \alpha\in I\}. \]
The constant group scheme associated to $I^{-1}/\mathcal{O}_K$ will be denoted by $\underline{I^{-1}/\mathcal{O}_K}$. 

If $R$ is a $K$-algebra, a level structure on $\Phi$ is an isomorphism $\iota\colon \underline{I^{-1}/\mathcal{O}_K}\cong \Phi[I]$.

Let $F^1_I$ be the functor which associates to every $K$-algebra $R$ the set of isomorphism classes of Drinfeld $\mathcal{O}_K$-modules of rank $1$ with level $I$ structure over $R$.
Let $K_{I,\infty}$ be the maximal abelian extension of $K$ unramified outside $I$ and totally split at $\infty$.

\begin{theorem}[Drinfeld]\label{universalitywithlevel}
    If $I$ has at least two different prime divisors, then $K_{I,\infty}$ represents the functor $F^1_I$.
\end{theorem}
\begin{proof}
    See \cite[Theorem 3.7.2]{Tha04}.
\end{proof}

\subsection{Algorithm}\label{subsecalgo}
Now we present an algorithm to compute Drinfeld singular moduli.
It is inspired by previous works on this topic, see in particular \cite{DuHa94, Mac10, Mac11}.
For an imaginary quadratic order $\mathcal{O}$, the Hilbert class polynomial $H_\mathcal{O}(X)$ is the minimal polynomial over $A$ of any singular moduli with CM by $\mathcal{O}$.

\begin{proof}[Proof of Theorem \ref{algorithm}]
    Write $\mathcal{O}=A[c_0]$ with $c_0$ as in \eqref{Hilbert odd}, \eqref{Hilbert even}, or \eqref{Hilbert insep} according to each case.
    Let $\Phi_T=T+a\tau+b\tau^2$ be a rank 2 Drinfeld module with complex multiplication by an order containing $\mathcal{O}$.
By \cite[Theorem 5.2.11]{Papi} this is the case if and only if $\Phi_T$ admits an isogeny of the form \begin{equation}\label{isogeny}
        \phi=c_0+c_1\tau+...+c_n\tau^n
    \end{equation} with $n=\deg N_{K/k}(c_0)$.
    We can assume assume that $a\neq0$, otherwise we know $j=0$ which has complex multiplication by the maximal order $\mathbb{F}_{q^2}[T]$ of $\mathbb{F}_{q^2}(T)$ and we set $P_{\mathbb{F}_{q^2}[T]}(X)=X$.
    Since Drinfeld modules corresponding to homothetic lattices have the same endomorphism rings, we can scale $\Phi_T$ so that without loss of generality $a = 1$; we then have $\phi_T = T +\tau +j^{-1}\tau^2$ with $j=j(\Phi_T)$.
    Comparing the coefficients of the powers of $\tau$ in the equation $\psi\Phi_T=\Phi_T\psi$, we arrive at the following system of equations:
    \begin{align}
        c_0 T &= T c_0 \label{E0} \tag{E$0$} \\
        c_0 + c_1 T^q &= c_0^q + c_1 T \label{E1} \tag{E$1$} \\
        c_{k-2} j^{-q^{k-2}} +c_{k-1} +c_k T^{q^k} &= j^{-1} c_{k-2}^{q^2} +c_{k-1}^q +T c_k \qquad\text{for} \ 2 \leqslant k \leqslant n \label{Ek} \tag{E$k$} \\
        c_{n-1} j^{-q^{n-1}} +c_n &= j^{-1} c_{n-1}^{q^2} +c_n^q \label{En+1} \tag{E$n+1$} \\
        c_n j^{-q^n} &= j^{-1} c_n^{q^2} \label{En+2} \tag{E$n+2$}
    \end{align}
    Beginning at \eqref{E1}, we can express $c_1$ in terms of $T$ and $c_0$; specifically there exists $C_1(Y)=\frac{Y^q-Y}{T^q-T} \in k[Y]$ such that $c_1=C_1(c_0)$.
   
    Using (E2) to (En), we recursively obtain $C_k(X,Y)\in k[X^{-1},Y]$ such that $c_k=C_k(j,c_0)$.
    Replacing $c_{n-1}$ and $c_n$ in \eqref{En+1} and \eqref{En+2}, we get two polynomials $P_i(X,Y)\in k[X^{-1},Y]$ such that the last two equations are of the form $P_i(j,c_0)=0$; clearing the powers of $X$ in the denominator of the greatest common divisor of $P_1(X,c_0)$ and $P_2(X,c_0)$ normalized to have constant term equal to $1$, we arrive at a monic polynomial $P_\mathcal{O}(X) \in K[X]$.
    If $\mathcal{O}\subsetneq\mathbb{F}_{q^2}[T]$, replace $P_\mathcal{O}(X)$ by $XP_\mathcal{O}(X)$.
    By construction, each singular moduli with complex multiplication by an order containing $\mathcal{O}$ satisfies $P_\mathcal{O}(j)=0$.
    Conversely, if $j\neq0$ is a root of the polynomial $P_\mathcal{O}$, then it satisfies both \eqref{En+1} and \eqref{En+2}, and hence satisfies $\psi \phi_T = \phi_T \psi$ with $\phi$ of the form \eqref{isogeny} and $c_k=C_k(j,c_0)$ ($1\leqslant k\leqslant n$); since the roots of $P_\mathcal{O}$ are integral, we conclude that $P_\mathcal{O}(X)\in \mathcal{O}_K[X]$.
    This proves item 1.

    Now we prove item 2.
    Observe that $C_1$ satisfies $C_1(-Y)=-C_1(Y)$ and $C_1(Y+b)=C_1(Y)$ for all $b$ in $k$.
    From (E2) to \eqref{En+2}, one can recursively see that $C_k$ ($2\leqslant k\leqslant n$) and $P_i(X,Y)$ ($i=1,2$) inherit these properties.
    Assume first that $q$ is odd.
    Specializing to $Y=c_0=\sqrt{D}$ we obtain from $P_i(X,-c_0)=-P_i(X,c_0)$ that $P_i(X,c_0) \in c_0 k[X]$.
    In particular, when considering the greatest common divisor of $P_1(X,c_0)$ and $P_2(X,c_0)$, we obtain that $P_\mathcal{O}\in \mathcal{O}_K[X]\cap k[X]=A[X]$.
    Now we consider the even characteristic case.
    This time we specialize to $Y=c_0$ a root of $X^2+bX+c$ with $b,c\in A\smallsetminus\{0\}$.
    Then $\sigma(c_0)=c_0+b$ is the Galois conjugate of $c_0$.
    Note that $\mathcal{O}=A[c_0]=A[\sigma(c_0)]$; thus $P_i(X,c_0)=0$ ($i=1,2$) if and only if $P_i(X,\sigma(c_0))=0$ ($i=1,2$).
    Since $P_i(X,\sigma(c_0))=P_i(X,c_0)+P_i(X,b)$, we must have $P_i(X,b)=0$.
    It follows that $P_i(X,c_0)$ is invariant under $\sigma$, \textit{i.e.} $P_i(X,c_0)$ belongs to $k[X^{-1}]$ and thus $P_\mathcal{O}(X)\in A[X]$.
    The last statement in item 2 is a consequence of $P_\mathcal{O}(X)$ having integral coefficients and of item 1.
    Indeed, let $\mathcal{O}_1,\dotsc,\mathcal{O}_m$ be an enumeration of the orders containing $\mathcal{O}$.
    There exists a positive integer $r(\mathcal{O}_1)$ such that $H_{\mathcal{O}_1}^{r(\mathcal{O}_1)}$ divides $P_\mathcal{O}(X)$ and the roots of $P_\mathcal{O}(X)/H_{\mathcal{O}_1}^{r(\mathcal{O}_1)}$ do not have CM by $\mathcal{O}_1$.
    With the same reasoning, there exists $r(\mathcal{O}_2)$ such that $H_{\mathcal{O}_2}^{r(\mathcal{O}_2)}$ divides $P_\mathcal{O}(X)/H_{\mathcal{O}_1}^{r(\mathcal{O}_1)}(X)$ and the roots of $P_\mathcal{O}(X)/(H_{\mathcal{O}_1}^{r(\mathcal{O}_1)}H_{\mathcal{O}_2}^{r(\mathcal{O}_2)}(X))$ do not have CM by $\mathcal{O}_1$ nor $\mathcal{O}_2$.
    Inductively, there exist positive integers $r(\mathcal{O}_i)$ for $1\leqslant i\leqslant m$ such that $P_\mathcal{O}(X) / (H_{\mathcal{O}_1}^{r(\mathcal{O}_1)} H_{\mathcal{O}_2}^{r(\mathcal{O}_2)} \dotsm H_{\mathcal{O}_m}^{r(\mathcal{O}_m)}(X))$ is constant.
    Since the Hilbert class polynomials and $P_\mathcal{O}$ are monic, this constant must be $1$.

   Item 3 is a direct consequence of the fact that in the inseparable case, $\mathcal{O}_K=\mathbb{F}_q[\sqrt{T}]$.
    Since $q$ is even, taking squares respects addition and then $P^2_\mathcal{O}\in \mathbb{F}_q[T]=A$.
    The product formula is proven exactly as in item 2.

Finally, we prove the last part of the statement.
In follows trivially that for $K=\F_{q^2}(T)$, one has $P_{\mathcal{O}_K}(X)=X=H_{\mathcal{O}_K}(X)$.
Let $K\neq\mathbb{F}_{q^2}(T)$ be an imaginary quadratic field and $\mathcal{O}_K$ its maximal order.
In the separate (resp. inseparable) case, denote $P_{\mathcal{O}_K}$ (resp. $P_{\mathcal{O}_K}^2$) simply by $P$. Then, we know that $P=H_{\mathcal{O}_K}^{r(\mathcal{O}_K)}$ for some positive integer $r(\mathcal{O}_K)$. Denote by $x$ the image of $X$ in the $K$-algebra $R=K[X]/P(X)$.
We will show that $H_{\mathcal{O}_K}(x)=0$ in $R$, which implies that $r(\mathcal{O}_K)=1$.
Since $(P,X)=1$, the element $x$ is invertible in $R$.
By construction of $P_{\mathcal{O}_K}(X)$, the rank 2 Drinfeld $A$-module over $R$ given by $\Phi_T=T+\tau+x^{-1}\tau^2$ has CM by $\mathcal{O}_K$.
In particular, the map $\mathcal{O}_K\to R\{\tau\}$ sending $T\mapsto\phi_T$ and $c_0\to\phi$ defines a rank 1 Drinfeld $\mathcal{O}_K$-module $\Phi$. 

If $K$ has a degree $1$ prime, by Theorem \ref{universality} there exists $\widetilde{\Phi}$, a Drinfeld $\mathcal{O}_K$-module of rank 1 over $\mathcal{O}_H$, and a morphism $\rho\colon \mathcal{O}_H\to R$ of $\mathcal{O}_H$-algebras sending $\widetilde{\Phi}$ to $\Phi$.
By definition, $j(\widetilde{\Phi}_T)$ satisfies the equation $H_{\mathcal{O}_K}(X)=0$.
In particular, $x=j(\Phi_T)=\rho(j(\widetilde{\Phi}_T))$ satisfies $H_{\mathcal{O}_K}(x)=0$ in $R$ and we are done.

In the general case, we can proceed as follows.
Let $\mathfrak{p} \in A$ be any prime that splits in $K$, then the ideal $I=\mathfrak{p}\mathcal{O}_K$ is the product of two different primes in $\mathcal{O}_K$.
Let $p$ be the monic generator of $\mathfrak{p}$, then $\Phi[I]=\mathrm{Spec}(R')$ with $R':=R[Z]/\Phi_p(Z)$.
Let $z$ be the image of $Z$ in $R'$. We claim that for $\alpha\in \mathcal{O}_K$, $\Phi_{\alpha}(z)=0$ if and only if $\alpha$ is in $I$.
We can easily see that if $\alpha=p\beta$ with $\beta$ in $\mathcal{O}_K$, then $\phi_{\alpha}(z)=\phi_{\beta}(\phi_p(z))=\phi_{\beta}(0)=0$ in $R'$.
On the other hand, $\Phi_\alpha(z)=0$ is equivalent to $\Phi_\alpha(Z)$ being divisible by $\Phi_p(Z)$ in $R[Z]$.
This implies that for any geometric point $R\to\overline{K}$, $\Phi[I](\overline{K})\subseteq \Phi[\alpha](\overline{K})$.
By \cite[Corollary A.15]{Papi} this is only possible if there exists an injective map from $\mathcal{O}_K/I$ to $\mathcal{O}_K/\alpha\mathcal{O}_K$.
This implies that $\alpha$ is divisible by the two primes of $K$ dividing $\mathfrak{p}$ and thus $\alpha$ is in $I$.
It follows that for every $\alpha\in I^{-1}/\mathcal{O}_K$, the set of well defined elements $\Phi_{p\alpha}(z)\in R'$ form the $\#\mathcal{O}/I=q^{2\deg p}=\deg_Z\Phi_p(Z)$ different roots of $\Phi_p(Z)$ in $R'$.
Thus, $\Phi_p$ splits completely in $R'$ and we have
\[ R'[W]/\Phi_p(W)\cong \prod_{\alpha\in I^{-1}/\mathcal{O}_K}R'[W]/(W-\phi_{p\alpha}(z))\cong\prod_{\alpha\in I^{-1}/\mathcal{O}_K}R'.\]
This isomorphism defines a level structure $\iota\colon \underline{I^{-1}/\mathcal{O}_K}\cong \Phi[I]$ over $R'$.
By Theorem \ref{universalitywithlevel}, there exists $(\widetilde{\Phi},\widetilde{\iota})$ a rank 1 Drinfeld $\mathcal{O}_K$-module over $K_{I,\infty}$ with $I$-level structure, and a morphism $\rho\colon K_{I,\infty}\to R'$ of $K$-algebras sending $(\widetilde{\Phi},\widetilde{\iota})$ to $(\Phi,\iota)$.
We know that $j(\widetilde{\Phi}_T)$ satisfies the equation $H_{\mathcal{O}_K}(X)=0$.
In particular, $x=j(\Phi_T)=\rho(j(\widetilde{\Phi}_T))$ satisfies $H_{\mathcal{O}_K}(x)=0$ in $R'$.
Since $R\to R'$ is injective, $H_{\mathcal{O}_K}(x)=0$ in $R$ and this finishes the proof.
\end{proof}

\begin{rem}
In the case where $K$ is inseparable, we opted to work with $P_{\mathcal{O}_K}^2$ to obtain a polynomial over $A$ since we are interested in the norm of the singular moduli.
Nonetheless, the same reasoning shows that $P_{\mathcal{O}_K}=\mathrm{irr}(j,\mathcal{O}_K)$ for any singular modulus $j$ with CM by $\mathcal{O}_K$, and in general, for an order $\mathcal{O}\subseteq\mathcal{O}_K$, the polynomial $P_\mathcal{O}$ satisfies an analogous product formula with respect to the minimal polynomials over $\mathcal{O}_K$ of singular moduli with CM by $\mathcal{O}'\supseteq \mathcal{O}$.
\end{rem}

\begin{rem} \label{rem:r=1}
    We have not been able to find in the literature a version of Theorem \ref{universalitywithlevel} addressing the case of non-maximal orders. If such a statement holds true, one could prove in a similar fashion that $r(\mathcal{O})=1$ for every order $\mathcal{O}$. Let us add that this has been the case in all our explicit computations, including for non-maximal orders.  
\end{rem}

Let us perform the computations in the case $\deg_T D = 1$ to illustrate the method.

\begin{prop}\label{n=1}
Let $q$ be odd and consider a discriminant $D = AT+B$ for $A,B \in \F_q$, and $A\neq0$.
Then the element $j = A^{-1}D^{\frac{q+1}{2}}(1+D^\frac{q-1}{2})^{q+1} \in \F_q[T]$ corresponds to a Drinfeld module of rank 2 with CM by $\sqrt{D}$.
\end{prop}

\begin{proof}
The equations used in the proof of Theorem \ref{algorithm} become
\begin{align}
    \sqrt{D} + c_1 T^q &= \sqrt{D}^q + c_1 T, \label{E11} \tag{E$1$} \\
    \sqrt{D} j^{-1} +c_1 &= j^{-1} {\sqrt{D}}^{q^2} +c_1^q, \label{E22} \tag{E$2$} \\
    c_1 j^{-q} &= j^{-1} c_1^{q^2}. \label{E33} \tag{E$3$}
\end{align}
From \eqref{E11}, we get $c_1 = \frac{(\sqrt{D}^q-\sqrt{D})}{T^q-T}$.
Using the relation $(\sqrt{D}^q-\sqrt{D})(\sqrt{D}^q+\sqrt{D})=D^q-D=A(T^q-T)$ we can rewrite $c_1 = \frac{A}{\sqrt{D}^q+\sqrt{D}}$.
Replacing in \eqref{E22} we obtain
% \begin{gather*}
%     \sqrt{D}+j\left(\frac{A}{\sqrt{D}^q+\sqrt{D}}\right)=\sqrt{D}^{q^2}+j\left(\frac{A}{\sqrt{D}^{q^2}+\sqrt{D}^q}\right) \\
%     \implies j =\frac{1}{A}(\sqrt{D}^{q^2}-\sqrt{D})\frac{(\sqrt{D}^q+\sqrt{D})^{q+1}}{\sqrt{D}^{q^2}+\sqrt{D}^q-\sqrt{D}^q-\sqrt{D}} = A^{-1} D^{\frac{q+1}{2}}(1+D^\frac{q-1}{2})^{q+1} \in \F_q[T].
%     \qedhere    
% \end{gather*}
% we do this so that the QED symbol will be flushed right correctly
\[
\begin{gathered}[b]
    \sqrt{D}+j\left(\frac{A}{\sqrt{D}^q+\sqrt{D}}\right)=\sqrt{D}^{q^2}+j\left(\frac{A}{\sqrt{D}^{q^2}+\sqrt{D}^q}\right) \\
    \implies j =\frac{1}{A}(\sqrt{D}^{q^2}-\sqrt{D})\frac{(\sqrt{D}^q+\sqrt{D})^{q+1}}{\sqrt{D}^{q^2}+\sqrt{D}^q-\sqrt{D}^q-\sqrt{D}} = A^{-1} D^{\frac{q+1}{2}}(1+D^\frac{q-1}{2})^{q+1} \in \F_q[T].
\end{gathered} \qedhere
\]
\end{proof}

\begin{rem}\label{prime D odd degree}
    From Proposition \ref{n=1} we see that there exist Drinfeld singular moduli associated to prime $D$ with odd degree $1$.
    However, these are not $S_0$-units, since for a prime factor $\mathfrak{p}$ of $j$,
    \[ \mathfrak{p} \mid j \implies \mathfrak{p} \mid D \left(1+D^{\frac{q-1}{2}}\right) \left(1-D^{\frac{q-1}{2}}\right) = D-D^q = A (T-T^q) = -A \prod_{\ell \in \F_q} (T-\ell), \]
    so $\mathfrak{p}$ is of odd degree $1$.
\end{rem}

Let us now exhibit an example where $\mathcal{O}$ is not maximal as an illustration of Theorem \ref{algorithm}.

\begin{ex}
    Let $q = 3$ and let $\mathcal{O} = Q[\sqrt{D}]$ where $D = T^3$.

    In this case, $n = 3$ and the polynomials $c_k(X)$ found in the execution of the algorithm are as follows:
    \begin{align*}
        c_0 &= \sqrt{D} \\
        c_1 &= \frac{T^2 + T + 1}{T^2 + T} \sqrt{D} \\
        c_2 &= \frac{T^8 + T^4 + 1}{T^5 + T} \sqrt{D} X + \frac{T^4 + T^3 + 1}{T^{10} + T^9 + T^7 + 2T^6 + T^5 + T^3 + T^2} \sqrt{D} \\
        c_3 &= \scalebox{0.95}{\begin{multlined}[t]
            \frac{T^{13} + T^{12} + 1}{T^{28} + T^{27} + T^{16} + 2T^{15} + T^{14} + T^3 + T^2} \sqrt{D} X^3 + \frac{T^{13} + T^9 + 1}{T^{28} + T^{24} + T^{19} + 2T^{15} + T^{11} + T^6 + T^2} \sqrt{D} X \\
            + \frac{2T^8 + T^3 + 2T^2 + T + 2}{T^{41} + T^{38} + T^{32} + 2T^{29} + T^{28} + T^{26} + T^{25} + T^{20} + T^{19} + T^{17} + 2T^{16} + T^{13} + T^7 + T^4} \sqrt{D}
        \end{multlined}}
    \end{align*}
    The polynomials $P_1(j,\sqrt{D})$ and $P_2(j,\sqrt{D})$ are too large to be displayed here.
    However, their greatest common divisor, $P_\mathcal{O}(j)$, is found to be the following:
    \begin{align*}
         &\mathrel{\phantom{=}} P_\mathcal{O}(j) \\
         &= \begin{multlined}[t]
             j^4 + (2T^{18} + 2T^{15} + 2T^{13} + T^{12} + 2T^{11} + T^9 + T^8 + 2T^6 + 2T^5 + 2T^3 + T) j^3 \\
             + (2T^{22} + T^{21} + 2T^{20} + T^{19} + 2T^{14} + T^9 + 2T^8 + T^7 + 2T^6) j^2 \\
             + 2 T^2 (T + 1)^{12} (T^{14} + T^{12} + 2T^{10} + T^6 + T^4 + 1) j + T^4 (T + 1)^{20} (T^3 + T^2 + 2T + 1)^4
         \end{multlined} \\
         &= \begin{multlined}[t]
             [j + 2 (T^6 + T^5 + T^3 + T^2)] \\
             \times \scalebox{0.85}{\begin{multlined}[t]
                 [j^3 + (2T^{18} + 2T^{15} + 2T^{13} + T^{12} + 2T^{11} + T^9 + T^8 + T^2 + T) j^2 \\
             + (2T^{24} + 2T^{23} + 2T^{22} + 2T^{21} + 2T^{18} + 2T^{17} + T^{16} + T^{15} + T^{14} + T^{12} + 2T^{11} + T^{10} + T^9 + T^5 + 2T^4 + T^3) j \\
             + 2T^{30} + T^{29} + 2T^{26} + T^{24} + 2T^{23} + 2T^{21} + 2T^{20} + 2T^{18} + 2T^{14} + 2T^9 + T^8 + 2T^6 + 2T^2]
             \end{multlined}}
         \end{multlined}
    \end{align*}
    As we see, the polynomial $P_\mathcal{O}$ factorises into two irreducible factors, of degree $1$ and $3$, being the Hilbert polynomials corresponding to the orders $\mathcal{O}' = A[\sqrt{T}] \supset A[\sqrt{T^3}]$ and $\mathcal{O} = A[\sqrt{T^3}]$ respectively.
    In both of these cases, $r(\mathcal{O}') = 1$.

    Note that the first factor gives a value of $j = T^6+T^5+T^3+T^2 = T^2 (1+T)^4$, which by Proposition \ref{n=1} is the $j$ invariant corresponding to a Drinfeld module of rank $2$ with CM by $\sqrt{T}$; the decomposition into two factors is as in Theorem \ref{algorithm}.
\end{ex}

\subsection{Explicit Drinfeld singular moduli}
The algorithm in Theorem \ref{algorithm} was implemented in SageMath \cite{sagemath} in a GitHub repository \cite{Code} (for recent developments concerning Drinfeld modules in SageMath, one may refer to \cite{ABCLNP26}).
In Figures
\ref{fig:compu_q3n2},
\ref{fig:compu_q3n3},
\ref{fig:compu_q3n4_squarefree},
\ref{fig:compu_q5n2},
\ref{fig:compu_q2n3},
and
\ref{fig:compu_q2n5}
we exhibit some tables of $P_{A[\sqrt{D}]}$ for small values of $q$ and squarefree $D$, and in Figures \ref{fig:compu_q3n3_squareful}, \ref{fig:compu_q3factor}, and \ref{fig:compu_q3n4_squareful} we exhibit some similar tables where $D$ is not squarefree; in each case, $P_\mathcal{O}$ was monic and integral in $\F_q(T,\sqrt{D})$, so to investigate for which sets $S$ of primes the element $j$ can be an $S$-unit, we need only find the prime factors of the (norm of the) constant coefficient of $P_\mathcal{O}$.
In addition, in the case $q=2$ the polynomials $P_\mathcal{O}$ found do not have coefficients in $\F_2[T]$, however $P_\mathcal{O}^2$ does.

A modified version of the algorithm in Theorem \ref{algorithm} was also implemented in SageMath in \cite{Code} where the polynomial $P_\mathcal{O}$ is computed modulo a set $\mathfrak{S}$ of `small primes' in $\F_q[T]$ and then combined using the Chinese Remainder Theorem, inspired by \cite{Sut11}.
The only conditions on $\mathfrak{S}$ are that each prime has degree greater than $n$, since to solve for each $c_k$ in terms of $c_{k-2}$ and $c_{k-1}$ one needs to divide by $T^{q^k}-T$ for each $k \leqslant n$, and that the (norm of the) product of the primes in $\mathfrak{S}$ has to be greater than that of the largest coefficient in $P_\mathcal{O}$; to estimate these coefficients, we used the descriptions of the sets $S_D$ in \cite[pp. 5-6]{AABP26} combined with Theorem \ref{Bro} due to Brown.

While we do not yet have computational complexity bounds, in practice this was found to be significantly faster than the straightforward calculation.
Given that the calculations of $P_\mathcal{O}$ modulo each prime are independent of each other, they can also be performed in parallel, yielding another speedup.

In addition, in Figure \ref{fig:compu_q2proper} we include some examples of $P_\mathcal{O}$ for the characteristic $2$ case \eqref{Hilbert even}.
Note that in these tables, since all the prime factors presented have odd degree, the $j$-invariants are all examples of $S$-units for $S$ a finite set of primes of odd degree, but none of them is an $S_0$-unit.

Since some $D$ are equivalent under automorphisms of $\F_q[T]$ like $T \mapsto aT+b$ for $a \in \F_q^\times$ and $b \in \F_q$, we include only one $D$ from each equivalence class.
Additionally, we exclude some $D$ which would not yield any imaginary quadratic extensions, as in \cite[Proposition 14.6]{Rosen}.

% Finally, we remark that in all computed examples where $D$ is squarefree, the polynomial $P_\mathcal{O}(j)$ is irreducible; hence in the cases calculated we have in fact found the Hilbert class polynomial.

Finally, our computations as in Figure \ref{fig:compu_q3n3}, Proposition \ref{n=1}, and Figure \ref{fig:compu_q2proper} agree with the results of Schweizer in \cite[Theorem 6]{sch97}.
Although it may seem as if there is a discrepancy in the value(s) of $j$ for $q=2$ and the order generated by $X$ with $X^2+TX+T^3 = 0$, this is due to this order not being maximal; the maximal order is generated by $X$ with $X^2+X+T = 0$, with $j$-invariant $(T+1)^3$ as in Figure \ref{fig:compu_q2proper}, and the remaining factor agrees with Schweizer's table.

We now list the below tables, each of which lists the polynomials $P_\mathcal{O}$, with a brief description of each:
\begin{description}
    \item[$q=3$:] Here $\mathcal{O} = A[\sqrt{D}]$; when
    \begin{description}
        \item[$D$ is squarefree:] Figures \ref{fig:compu_q3n2}, \ref{fig:compu_q3n3}, and \ref{fig:compu_q3n4_squarefree} for $\deg{D} = 2$, $3$, and $4$ respectively,
        \item[$D$ is not squarefree:] Figure \ref{fig:compu_q3n3_squareful} for $\deg{D} = 3$, and Figure \ref{fig:compu_q3n4_squareful} for an example where $\deg{D} = 4$ and $P_\mathcal{O}$ is irreducible,
        \item[$D$ has many factors:] Figure \ref{fig:compu_q3factor} gives an example where the order $\mathcal{O} = A[\sqrt{D}]$ is included in many other orders, illustrating the factorisation in Theorem \ref{algorithm}.
    \end{description}
    \item[$q=5$:] Here $\mathcal{O} = A[\sqrt{D}]$, for squarefree $\deg{D} = 2$, in Figure \ref{fig:compu_q5n2}, providing counterexamples to Dorman's conjecture.
    \item[$q=2$:] Here we either list $P_\mathcal{O}^2$ for $\mathcal{O} = A[\sqrt{D}]$ in Figures \ref{fig:compu_q2n3} and \ref{fig:compu_q2n5} for $\deg{D} = 3$ and $5$ respectively, or we list $P_\mathcal{O}$ for $\mathcal{O} = A[f\sqrt{D}\xi]$ in Figure \ref{fig:compu_q2proper}.
\end{description}

\renewcommand{\arraystretch}{1.2}

\begin{figure}[H]
    \centering
    \begin{tabular}{|c|p{0.5\textwidth}|}
        \hline
        $D$ ($q = 3$,   & $P_\mathcal{O}$ (with constant coefficients factorised) \\
        $\deg{D} = 2$)  & \\ \hline
        $2T^2+1$    & $j^2 + (T^9 + 2T^7 + T^5 + 2T)j + (T+1)^4(T+2)^4$ \\ \hline
        $2T^2 + 2$  & $j^2 + (T^9 + 2T^7 + 2T^5)j + 2T^8$ \\ \hline
    \end{tabular}
    \caption{Table of $P_\mathcal{O}$ for $q=3$, $\mathcal{O} = A[\sqrt{D}]$, and $D$ quadratic and squarefree}
    \label{fig:compu_q3n2}
\end{figure}

\begin{figure}[H]
    \centering
    \begin{tabular}{|c|p{0.71\textwidth}|}
        \hline
        $D$ ($q = 3$, $\deg{D} = 3$)    & $P_\mathcal{O}$ (with only leading and constant coefficients included and factorised) \\ \hline
        $T^3 + T$               & $ j^4 + \dotsb + T^8 (T + 2)^{16} (T^3 + T^2 + 2T + 1)^4$ \\ \hline
        $T^3 + 2T$              & $j^4 + \dotsb + T^8 (T + 1)^8 (T + 2)^8 (T^3 + 2T + 1)^4$ \\ \hline
        $T^3 + 2T + 1$          & $\begin{multlined}[t]
            j^7 + \dotsb + 2 (T^3 + T^2 + T + 2)^4 (T^3 + T^2 + 2T + 1)^4 \\
            \times (T^3 + 2T + 1)^2 (T^3 + 2T + 2)^4 (T^3 + T^2 + 2)^4
        \end{multlined}$ \\ \hline
        $T^3 + 2T + 2$          & $j + 2 T^4 (T + 1)^4 (T + 2)^4 (T^3 + 2T + 2)^2$ \\ \hline
        $T^3 + T^2 + 1$         & $j^6 + \dotsb + (T + 2)^{12} (T^3 + T^2 + 2)^4 (T^3 + 2T^2 + 1)^4 (T^3 + 2T^2 + 2T + 2)^4$ \\ \hline
        $T^3 + T^2 + 2$         & $j^3 + \dotsb + 2 T^{12} (T + 1)^{12} (T^3 + T^2 + 2)^2$ \\ \hline
        $T^3 + 2T^2 + 1$        & $j^5 + \dotsb + 2 (T + 1)^{24} (T^3 + 2T^2 + 1)^2 (T^3 + 2T + 2)^4$ \\ \hline
        $T^3 + 2T^2 + 2$        & $j^2 + \dotsb + T^8 (T+1)^4 (T + 2)^{12}$ \\ \hline
    \end{tabular}
    \caption{Table of $P_\mathcal{O}$ for $q=3$, $\mathcal{O} = A[\sqrt{D}]$, and $D$ cubic and squarefree}
    \label{fig:compu_q3n3}
\end{figure}

\begin{figure}[H]
    \centering
    \begin{tabular}{|c|p{0.7\textwidth}|}
        \hline
        $D$ ($q = 3$, $\deg{D} = 4$)    & $P_\mathcal{O}$ (with leading and constant coefficients included and factorised) \\ \hline
        $2T^4+1$    & $j^8 + \dotsb + (T + 1)^{16} (T + 2)^{16} (T^3 + T^2 + 2T + 1)^4 (T^3 + 2T^2 + 2T + 2)^4$ \\ \hline
        $2T^4+2$    & $j^8 + \dotsb + T^{32} (T^3 + T^2 + 2)^4 (T^3 + 2T^2 + 1)^4$ \\ \hline
        $2T^4+T+1$  & $j^8 + \dotsb + 2 (T + 1)^{36} (T^3 + 2T + 2)^4 (T^3 + 2T^2 + 2T + 2)^4$ \\ \hline
        $2T^4+T+2$  & $j^2 + \dotsb + T^{12} (T + 1)^4 (T + 2)^8 (T^3 + T^2 + T + 2)^4$ \\ \hline
        $2T^4+T^2+1$    & $\begin{multlined}[t]
            j^{12} + \dotsb + 2 (T^3 + 2T + 1)^4 (T^3 + 2T + 2)^4 (T^3 + T^2 + 2)^4 \\
            \times (T^3 + T^2 + T + 2)^4 (T^3 + 2T^2 + 1)^4 (T^3 + 2T^2 + T + 1)^4
        \end{multlined}$ \\ \hline
        $2T^4+T^2+T$    & $j^6 +\dotsb + T^{12} (T + 1)^{28} (T^3 + T^2 + 2T + 1)^4$ \\ \hline
        $2T^4+T^2+T+2$  & $j^6 +\dotsb + T^{28} (T + 2)^{12} (T^3 + T^2 + 2T + 1)^4$ \\ \hline
        $2T^4+2T^2+1$   & $j^4 +\dotsb + 2(T + 1)^{20} (T + 2)^{20}$ \\ \hline
        $2T^4+2T^2+T$   & $j^4 +\dotsb + T^8 (T + 1)^8 (T + 2)^{16} (T^3 + T^2 + 2)^4$ \\ \hline
        $2T^4+2T^2+T+1$ & $\begin{multlined}[t]
            j^{10} +\dotsb +(T + 2)^{20} (T^3 + 2T + 2)^4 (T^3 + T^2 + 2)^4 \\
            \times (T^3 + 2T^2 + T + 1)^4 (T^3 + 2T^2 + 2T + 2)^4
        \end{multlined}$ \\ \hline
        $2T^4+2T^2+T+2$ & $j^4 +\dotsb + 2T^{16} (T + 1)^{16} (T^3 + 2T + 1)^4$ \\ \hline
    \end{tabular}
    \caption{Table of $P_\mathcal{O}$ for $q=3$, $\mathcal{O} = A[\sqrt{D}]$, and $D$ quartic and squarefree}
    \label{fig:compu_q3n4_squarefree}
\end{figure}

\begin{figure}[H]
    \centering
    \begin{tabular}{|c|p{0.71\textwidth}|}
        \hline
        $D$ ($q = 3$, $\deg{D} = 3$)    & $P_\mathcal{O}$ (with only leading and constant coefficients included and factorised) \\ \hline
        $T^3$       & $(j+2T^2(T+1)^4) \times (j^3 + \dotsb + 2 T^2 (T + 1)^{16} (T^3 + T^2 + 2T + 1)^4)$ \\ \hline
        $T^3+T^2$   & $(j + 2 (T + 1)^2 (T + 2)^4) \times (j^2 +\dotsb + (T + 1)^4 (T + 2)^8 (T^3 + 2T^2 + T + 1)^4)$ \\ \hline
        $T^3+2T^2$  & $(j + 2 T^4 (T+2)^2) \times (j^3 +\dotsb + T^4 (T + 2)^8 (T^3 + 2T + 1)^4 (T^3 + 2T^2 + 1)^4)$ \\ \hline
    \end{tabular}
    \caption{Table of $P_\mathcal{O}$ (all not irreducible) for $q=3$, $\mathcal{O} = A[\sqrt{D}]$, and $D$ cubic and not squarefree}
    \label{fig:compu_q3n3_squareful}
\end{figure}

\begin{figure}[H]
    \centering
    \begin{tabular}{|c|p{0.71\textwidth}|}
        \hline
        $D$ ($q = 3$, $\deg{D} = 4$)    & $P_\mathcal{O}$ (with leading and constant coefficients included and factorised) \\ \hline
        $\begin{gathered}[t]
            2T^4 + T^2 + 2 \\
            = 2(T^2+1)^2
        \end{gathered}$ & $\begin{multlined}[t]
            j \times [j^2 + (T^{27} + 2T^{21} + 2T^{19} + 2T^{17} + T^{15} + 2T^{13} + 2T^9 + 2T^7 + T^5) j \\
            + 2 T^8 (T + 1)^8 (T + 2)^8 (T^2 + 1)]
        \end{multlined}$ \\ \hline
    \end{tabular}
    \caption{Example of $P_\mathcal{O}$ for $q=3$, $\mathcal{O} = A[\sqrt{D}]$, and a selected $D$ quartic and not squarefree}
    \label{fig:compu_q3n4_squareful}
\end{figure}

\begin{figure}[H]
    \centering
    \begin{tabular}{|c|p{0.72\textwidth}|}
        \hline
        $D$ ($q=3$) & $P_{\mathcal{O}}$ (factorised, with leading and constant coefficients included) \\ \hline
        $T$                 & $j + 2T^2(T+1)^4$ \\ \hline
        $T(T+2)^2$          & $\bigl[j + 2T^2(T+1)^4\bigr] \times \bigl[j^2 + \dotsb + T^4 (T + 1)^8 (T^3 + 2T^2 + 1)^4\bigr]$ \\ \hline
        $T(T+1)^2$          & $\begin{multlined}[t]
            \bigl[j + 2T^2(T+1)^4\bigr] \\
            \times \bigl[j^4 + \dotsb + T^8 (T + 1)^4 (T^3 + 2T + 1)^4 (T^3 + 2T^2 + T + 1)^4\bigr]
        \end{multlined}$ \\ \hline
        $T(T+1)^2(T+2)^2$   & $\begin{multlined}[t]
            \bigl[j + 2T^2(T+1)^4\bigr] \times \bigl[j^2 + \dotsb + T^4 (T + 1)^8 (T^3 + 2T^2 + 1)^4\bigr] \\
            \times \Bigl[j^4 + \dotsb + T^8 (T + 1)^4 (T^3 + 2T + 1)^4 (T^3 + 2T^2 + T + 1)^4\Bigr] \\
            \times \begin{multlined}[t]
                \left[j^8 + \dotsb + T^{16} (T + 1)^8 (T^3 + T^2 + 2T + 1)^8 (T^5 + T^3 + T + 1)^4 \right. \\
                \left. \times (T^5 + T^4 + T^3 + 2T^2 + T + 1)^4 (T^5 + T^4 + 2T^3 + 2T^2 + 1)^4\right]
            \end{multlined}
        \end{multlined}$ \\ \hline
    \end{tabular}
    \caption{Table of $P_{A[\sqrt{D}]}$ for $q=3$ and selected $D$ to illustrate the factorisation in Theorem \ref{algorithm}}
    \label{fig:compu_q3factor}
\end{figure}

\begin{figure}[H]
    \centering
    \begin{tabular}{|c|p{0.84\textwidth}|}
        \hline
        $D$ ($q = 5$,   & $P_\mathcal{O}$ (with constant coefficient factorised) \\
        $\deg{D} = 2$)  & \\ \hline
        $2T^2 + 1$  & $j^2 + (T^{25} + 4T^{21} + 3T^{19} + 4T^{17} + 4T^{15} + 4T^{11} + 2T^9 + T^7 + 2T) j + 3 (T+1)^{12} (T+4)^{12}$ \\ \hline
        $2T^2 + 2$  & $j^2 + (T^{25} + 4T^{21} + T^{19} + T^{17} + 2T^{15} + 3T^{11} + 3T^9 + 3T^7) j + T^{12} (T + 2)^6 (T + 3)^6$ \\ \hline
        $2T^2 + 3$  & $j^2 + (T^{25} + 4T^{21} + 4T^{19} + T^{17} + 3T^{15} + 2T^{11} + 3T^9 + 2T^7) j + 4 T^{12} (T + 1)^6 (T + 4)^6$ \\ \hline
        $2T^2 + 4$  & $j^2 + (T^{25} + 4T^{21} + 2T^{19} + 4T^{17} + T^{15} + T^{11} + 2T^9 + 4T^7 + 2T) j + 2 (T + 2)^{12} (T + 3)^{12}$ \\ \hline
    \end{tabular}
    \caption{Table of $P_\mathcal{O}$ for $q=5$, $\mathcal{O} = A[\sqrt{D}]$, and $D$ quadratic and squarefree}
    \label{fig:compu_q5n2}
\end{figure}

\begin{figure}[H]
    \centering
    \begin{tabular}{|c|p{0.35\textwidth}|}
        \hline
        $D$ ($q=2$,     & $P_\mathcal{O}^2$ (with only leading and constant \\
        $\deg{D} = 3$)  & coefficients included and factorised)\\ \hline
        $T^3 + 1$       & $j^6 + \dotsb + T^6 (T + 1)^9 (T^3 + T^2 + 1)^3$ \\ \hline
        $T^3 + T + 1$   & $j^6 + \dotsb + T^9 (T + 1)^6 (T^3 + T + 1)^3$ \\ \hline
    \end{tabular}
    \caption{Table of $P_\mathcal{O}^2$ for $q=2$, $\mathcal{O} = A[\sqrt{D}]$, and $D$ cubic and squarefree (inseparable case)}
    \label{fig:compu_q2n3}
\end{figure}

\begin{figure}[H]
    \centering
    \begin{tabular}{|c|p{0.73\textwidth}|}
        \hline
        $D$ ($q=2$, $\deg{D} = 5$)  & $P_\mathcal{O}^2$ (with only leading and constant coefficients included and factorised) \\ \hline
        $T^5 + 1$       & $j^{14} + \dotsb + T^9 (T + 1)^{21} (T^3 + T^2 + 1)^3 (T^3 + T + 1)^6 (T^5 + T^2 + 1)^3$ \\ \hline
        $T^5 + T + 1$   & $j^{14} + \dotsb + T^{21} (T + 1)^9 (T^3 + T + 1)^3 (T^3 + T^2 + 1)^6 (T^5 + T^4 + T^2 + T + 1)^3$ \\ \hline
        $T^5 + T^2 + 1$ & $j^{14} + \dotsb + T^9 (T + 1)^{21} (T^3 + T^2 + 1)^3 (T^3 + T + 1)^6 (T^5 + T^2 + 1)^3$ \\ \hline
        $T^5 + T^2 + T$ & $j^{14} + \dotsb + T^{21} (T + 1)^9 (T^3 + T + 1)^3 (T^3 + T^2 + 1)^6 (T^5 + T^4 + T^2 + T + 1)^3$ \\ \hline
        $T^5 + T^3 + 1$ & $\begin{multlined}[t]
            j^{18} + \dotsb + T^{18} (T + 1)^{18} (T^3 + T + 1)^3 (T^3 + T^2 + 1)^3 (T^5 + T^3 + 1)^3 \\
            \times (T^5 + T^4 + T^3 + T^2 + 1)^3
        \end{multlined}$ \\ \hline
        $T^5 + T^3 + T + 1$         & $j^{10} + \dotsb + T^{15} (T + 1)^{15} (T^5 + T^3 + T^2 + T + 1)^3 (T^5 + T^4 + T^3 + T + 1)^3$ \\ \hline
        $T^5 + T^3 + T^2 + T$       & $j^{10} + \dotsb + T^{15} (T + 1)^{15} (T^5 + T^3 + T^2 + T + 1)^3 (T^5 + T^4 + T^3 + T + 1)^3$ \\ \hline
        $T^5 + T^3 + T^2 + T + 1$   & $j^{10} + \dotsb + T^{15} (T + 1)^{15} (T^5 + T^3 + T^2 + T + 1)^3 (T^5 + T^4 + T^3 + T + 1)^3$ \\ \hline
    \end{tabular}
    \caption{Table of $P_\mathcal{O}^2$ for $q=2$, $\mathcal{O} = A[\sqrt{D}]$, and $D$ quintic and squarefree (inseparable case)}
    \label{fig:compu_q2n5}
\end{figure}

\begin{figure}[H]
    \centering
    \begin{tabular}{|c|p{0.66\textwidth}|}
        \hline
        Minimal polynomial of $f\sqrt{D}\xi$ & $P_{\mathcal{O}}$ (constant term factorised) \\
        for $\mathcal{O}=A[f\sqrt{D}\xi]$ ($q=2$) & \\ \hline
        $X^2+X+1+T^3$           & $j^3 + (T^6 + T^5 + T^4) j^2 + (T^7 + T^6) j + T^{12}$ \\ \hline
        $\begin{gathered}[t]
            X^2+T(1+T)X \\
            +T(1+T)(1+T+T^2)
        \end{gathered}$ & $\begin{multlined}[t]
            j^4 + (T^8 + T^6 + T^5 + T^3) j^3 + (T^{10} + T^9 + T^8 + T^6 + T^5 + T^4) j^2 \\
            + (T^{16} + T^{14} + T^{13} + T^{11} + T^9 + T^8 + T^7 + T^6) j + T^6(T+1)^6
        \end{multlined}$ \\ \hline
        $\begin{gathered}[t]
            X^2+T(1+T)X \\
            +T(1+T)(1+T^2+T^3)
        \end{gathered}$ & $\begin{multlined}[t]
            j^4 + (T^{12} + T^8 + T^7 + T^6 + T^5 + T^4) j^3 \\
            + (T^{14} + T^{12} + T^{11} + T^7 + T^6 + T^4) j^2 \\
            + (T^{24} + T^{21} + T^{19} + T^{15} + T^{14} + T^{13} + T^{12} + T^{11} + T^{10} + T^7) j \\
            + T^6 (T + 1)^6 (T^5 + T^4 + T^3 + T^2 + 1)^3
        \end{multlined}$ \\ \hline
        $\begin{gathered}[t]
            X^2 +(T^2+T+1)X \\
            +(T^2+T+1)^2
        \end{gathered}$ & $j + T^3 (T + 1)^3 (T^2 + T + 1)$ \\ \hline
        $X^2+X+T^3+T+1$ & $j + T^3 (T + 1)^3$ \\ \hline
        $X^2+X+T^5+T^3+1$   & $j + T^6 (T+1)^6$ \\ \hline
        $X^2+TX+T^3$    & $(j+(T+1)^3) (j+(T+1)^6)$ \\ \hline
        $X^2+X+T$       & $j+(T+1)^3$ \\ \hline
        $X^2+X+T$ ($q=4$)   & $j + (T^2+T+1)^5$ \\ \hline
        $X^2+X+T$ ($q=8$)   & $j + (T+1)^9 (T^3+T^2+1)^9$ \\ \hline
    \end{tabular}
    \caption{Table of $P_\mathcal{O}$ for $q=2$ and $\mathcal{O}=A[f\sqrt{D}\xi]$ (separable case)}
    \label{fig:compu_q2proper}
\end{figure}

%%%%%%%%%%%%%%%%%%%%%%%%%%%%%%%%%%%%%%%%%%%%%%%%%%%%%%%%%%%%%%%%%%%%%%%%%%%%%%%%%%%%%%%%%%%%%%%%%%%%%%%%%%%%%%%%%%%%%%%%%%%%%%%%%%%%
%%%%%%%%%%%%%%%%%%%%%%%%%%%%%%%%%%%%%%%%%%%%%%%%%%%%%%%%%%%%%%%%%%%%%%%%%%%%%%%%%%%%%%%%%%%%%%%%%%%%%%%%%%%%%%%%%%%%%%%%%%%%%%%%%%%%
%%%%%%%%%%%%%%%%%%%%%%%%%%%%%%%%%%%%%%%%%%%%%%%%%%%%%%%%%%%%%%%%%%%%%%%%%%%%%%%%%%%%%%%%%%%%%%%%%%%%%%%%%%%%%%%%%%%%%%%%%%%%%%%%%%%%
\centerline{\rule{10cm}{.5pt}}
%%%%%%%%%%%%%%%%%%%%%%%%%%%%%%%%%%%%%%%%%%%%%%%%%%%%%%%%%%%%%%%%%%%%%%%%%%%%%%%%%%%%%%%%%%%%%%%%%%%%%%%%%%%%%%%%%%%%%%%%%%%%%%%%%%%%
%%%%%%%%%%%%%%%%%%%%%%%%%%%%%%%%%%%%%%%%%%%%%%%%%%%%%%%%%%%%%%%%%%%%%%%%%%%%%%%%%%%%%%%%%%%%%%%%%%%%%%%%%%%%%%%%%%%%%%%%%%%%%%%%%%%%

\paragraph{Acknowledgments --}
This project emerged under the umbrella of a \emph{GandA Research In Pairs} visit of FP and PP to LB. 
This visit was made possible through funding from IRN GandA (CNRS) and the Department of Mathematical Sciences from Stellenbosch University, which the authors heartily thank.
The authors thank Valentijn Karemaker and Mihran Papikian for useful discussions around the results of \cite{KaremakerKatenPapikian}; Florian Breuer, Antoine Leudi\`ere, and Damien Robert for feedback on Hilbert modular polynomials; and Emmanuel Reinecke for conversations on universality.
FP and PP wish to thank Stellenbosch University for the hospitality during their stay.

\centerline{\rule{10cm}{.5pt}}

%%%%%%%%%%%%%%%%%%%%%%%%%%%%%%%%%%%%%%%%%%%%%%%%%%%%%%%%%%%%%%%%%%%%%%%%%%%%%%%%%%%%%%%%%%%%%%%%%%%%%%%%%%%%%%%%%%%%%%%%%%%%%%%%%%%%
%%%%%%%%%%%%%%%%%%%%%%%%%%%%%%%%%%%%%%%%%%%%%%%%%%%%%%%%%%%%%%%%%%%%%%%%%%%%%%%%%%%%%%%%%%%%%%%%%%%%%%%%%%%%%%%%%%%%%%%%%%%%%%%%%%%%

%% \nocite{*}
%\bibliography{refs.bib}
\printbibliography{}
%%%%%%%%%%%%%%%%%%%%%%%%%%%%%%%%%%%%%%%%%%%%%%%%%%%%%%%%%%%%%%%%%%%%%%%%%%%%%%%%%%%%%%%%%%%%%%%%%%%%%%%%%%%%%%%%%%%%%%%%%%%%%%%%%%%%
%%%%%%%%%%%%%%%%%%%%%%%%%%%%%%%%%%%%%%%%%%%%%%%%%%%%%%%%%%%%%%%%%%%%%%%%%%%%%%%%%%%%%%%%%%%%%%%%%%%%%%%%%%%%%%%%%%%%%%%%%%%%%%%%%%%%
%%%%%%%%%%%%%%%%%%%%%%%%%%%%%%%%%%%%%%%%%%%%%%%%%%%%%%%%%%%%%%%%%%%%%%%%%%%%%%%%%%%%%%%%%%%%%%%%%%%%%%%%%%%%%%%%%%%%%%%%%%%%%%%%%%%%
	\normalsize\vfill
	\noindent\rule{7cm}{0.5pt}
 
    \smallskip
	\noindent
	{Liam {\sc Baker}} {(\tt \href{liambaker@sun.ac.za}{liambaker@sun.ac.za})}  --
    {\sc Department of Mathematical Sciences, Stellenbosch University}, Stellenbosch, South Africa 7599.

	\medskip
	\noindent
	{Fabien {\sc Pazuki}} {(\tt \href{fpazuki@math.ku.dk}{fpazuki@math.ku.dk})} --
	{\sc Department of Mathematical Sciences, University of Copenhagen}.
	Universitetsparken 5, 2100 Copenhagen \o{} (Denmark).

    \medskip
	\noindent
	{Patricio {\sc P{\'e}rez-Pi{\~n}a}} {(\tt \href{papp@math.ku.dk}{papp@math.ku.dk})}  --
	{\sc Department of Mathematical Sciences, University of Copenhagen}.
	Universitetsparken 5, 2100 Copenhagen \o{} (Denmark).

%%%%%%%%%%%%%%%%%%%%%%%%%%%%%%%%%%%%%%%%%%%%%%%%%%%%%%%%%%%%%%%%%%%%%%%%%%%%%%%%%%%%%%%%%%%%%%%%%%%%%%%%%%%%%%%%%%%%%%%%%%%%%%%%%%%%
%%%%%%%%%%%%%%%%%%%%%%%%%%%%%%%%%%%%%%%%%%%%%%%%%%%%%%%%%%%%%%%%%%%%%%%%%%%%%%%%%%%%%%%%%%%%%%%%%%%%%%%%%%%%%%%%%%%%%%%%%%%%%%%%%%%%
%%%%%%%%%%%%%%%%%%%%%%%%%%%%%%%%%%%%%%%%%%%%%%%%%%%%%%%%%%%%%%%%%%%%%%%%%%%%%%%%%%%%%%%%%%%%%%%%%%%%%%%%%%%%%%%%%%%%%%%%%%%%%%%%%%%%
%%%%%%%%%%%%%%%%%%%%%%%%%%%%%%%%%%%%%%%%%%%%%%%%%%%%%%%%%%%%%%%%%%%%%%%%%%%%%%%%%%%%%%%%%%%%%%%%%%%%%%%%%%%%%%%%%%%%%%%%%%%%%%%%%%%%
\end{document}